\documentclass[12pt,twoside,leqno]{article}
\usepackage[T1]{fontenc}
\usepackage{amsmath,amsthm}
\usepackage{amssymb,latexsym}
\usepackage{enumerate}

\theoremstyle{plain}
\newtheorem{thm}{Theorem}[section]

\newtheorem{prop}[thm]{Proposition}
\newtheorem{cor}[thm]{Corollary}

\newtheorem{f}[thm]{Fact}
\theoremstyle{definition}
\newtheorem{exam}[thm]{Example}
\newtheorem{defi}[thm]{Definition}
\theoremstyle{remark}
\newtheorem{rem}[thm]{Remark}

\newtheorem{p}[thm]{Problem}

\newcommand{\mb}{\mathbb}
\newcommand{\mc}{\mathcal}

\begin{document}

\title{Quasi-Metrizability of \\
Bornological Biuniverses in ZF}
\author{Artur Pi\k{e}kosz \\
Institute of Mathematics\\
Cracow University of Technology\\
Warszawska 24, \\
31-155 Cracow, Poland\\
E-mail: pupiekos@cyfronet.pl\\
and\\
Eliza Wajch\\
Institute of Mathematics and Physics\\
Siedlce University of Natural Sciences and Humanities\\
3 Maja 54,\\
08-110 Siedlce, Poland\\
E-mail: eliza.wajch@wp.pl }

\maketitle

\begin{abstract}
 Hu's metrization theorem for bornological universes is shown to hold in \textbf{ZF} and it is adapted to a quasi-metrization theorem for bornologies in bitopological spaces. The problem of uniform quasi-metrization of quasi-metric bornological universes is investigated.  Several consequences for natural bornologies in generalized topological spaces in the sense of Delfs and
Knebusch are deduced. Some statements concerning (uniform)-(quasi)-metrization of bornologies are shown to be relatively independent of \textbf{ZF}.  
\end{abstract}
\renewcommand{\thefootnote}{}



\renewcommand{\thefootnote}{\arabic{footnote}} \setcounter{footnote}{0}

\section{Introduction}

A \textbf{bitopological space} is a triple $(X, \tau_1, \tau_2)$ where $X$ is a set and $\tau_1, \tau_2$ are topologies in $X$. A \textbf{quasi-pseudometric} in a set $X$ is a function $d: X\times X\to [0; +\infty)$ such that, for all $x,y,z\in X$, $d(x,y)\leq d(x,z)+d(z,y)$ and $d(x,x)=0$. A quasi-pseudometric $d$ in $X$ is called a \textbf{quasi-metric} if, for all $x,y\in X$, the condition $d(x,y)=0$ implies $x=y$ (cf. \cite{Kel}, \cite{FL}). 

Let $d$ be a quasi-pseudometric in $X$. \textbf{The conjugate} of $d$ is the quasi-pseudometric $d^{-1}$ defined by $d^{-1}(x, y)=d(y, x)$ for $x, y\in X$. The $d$-ball with centre $x\in X$ and radius $r\in(0; +\infty)$ is the set $B_{d}(x, r)=\{ y\in X: d(x, y)<r\}$. The collection $\tau(d)=\{ V\subseteq X: \forall_{x\in V}\exists_{n\in\omega} B_{d}(x, \frac{1}{2^n})\subseteq V\}$ is \textbf{ the topology in $X$ induced by $d$}. The triple $(X,\tau(d), \tau(d^{-1}))$ is \textbf{ the bitopological space associated with $d$}. 
\begin{defi}
A bitopological space $(X, \tau_1, \tau_2)$ is \textbf{(quasi)-metrizable} if there exists a (quasi)-metric $d$ in $X$ such that $\tau_1=\tau(d)$ and $\tau_2=\tau(d^{-1})$ (cf. pp. 74--75 of \cite{Kel}).
\end{defi}
One can find a considerable number of quasi-metrization theorems in \cite{An} and in other sources (cf. \cite{FL}). 

We recall that, according to \cite{Al}--\cite{Hu}, a \textbf{boundedness} in a set $X$ is a (non-void) ideal of subsets of $X$. 
A boundedness $\mc{B}$ in $X$ is called a \textbf{bornology} in $X$ if every singleton of $X$ is a member of $\mc{B}$ (cf. 1.1.1 in \cite{H-N}).   

\begin{defi}[cf. Definition 4.1 of \cite{Hu}] If $\mc{B}$ is a boundedness in $X$, then a collection $\mc{A}$ is called \textbf{a base} for $\mc{B}$ if $\mc{A}\subseteq \mc{B}$ and every set of $\mc{B}$ is a subset of a member of $\mc{A}$. \textbf{A second-countable boundedness} is a boundedness which has a countable base.
\end{defi} 

\begin{defi} Let $(X, \tau_{1}, \tau_{2})$ be a bitopological space. A boundedness $\mc{B}$ in $X$ will be called $(\tau_{1}, \tau_{2})$\textbf{-proper} if, for each $A\in\mc{B}$, there exists $B\in\mc{B}$ such that $\text{cl}_{\tau_2}A\subseteq \text{int}_{\tau_1}(B)$. If $\tau=\tau_1=\tau_2$ and the boundedness $\mc{B}$ is $(\tau,\tau)$-proper, we will say that $\mc{B}$ is $\tau$\textbf{-proper}.
\end{defi}

Let us notice that if $(X, \tau)$ is a topological space, then a boundedness $\mc{B}$ in $X$ is $\tau$-proper if and only if the universe $((X, \tau),\mc{B})$ is proper in the  sense of Definition 3.4 of \cite{Hu}.

\begin{defi}
\begin{enumerate}
\item[(i)] We say that \textbf{a bornological biuniverse} is an ordered pair $((X, \tau_1, \tau_2), \mc{B})$  where $(X, \tau_1, \tau_2)$ is a bitopological space and $\mc{B}$ is a bornology in $X$ . 
\item[(ii)]A \textbf{bornological universe} is an ordered pair $((X, \tau), \mc{B})$ where $(X, \tau)$ is a topological space and $\mc{B}$ is a bornology in $X$ (cf. Definition 1.2 of \cite{Hu}).
\end{enumerate}
\end{defi}

\begin{defi} Let $d$ be a quasi-metric in $X$ and let $A$ be a subset of $X$. Then:
\begin{enumerate}
\item[(i)]  $A$ is  called \textbf{$d$-bounded} if there exist $x\in X$ 
and $r\in (0; +\infty)$ such that $A\subseteq B_d(x, r)$; 
\item[(ii)] if $A$ is not $d$-bounded, we say that $A$ is \textbf{$d$-unbounded};
\item[(iii)] $\mc{B}_d(X)$ is the collection of all d-bounded subsets of $X$.
\end{enumerate}
\end{defi}

For a quasi-metric $d$ in $X$, a set $A\subseteq X$ can be simultaneously $d$-bounded and $d^{-1}$-unbounded.

\begin{exam} For $x,y\in\omega$, let $d(x, y)=0$ if $x=y$ and $d(x, y)= 2^{x}$ if $x\neq y$. Then $\omega= B_d(0, 2)$, so $\omega$ is $d$-bounded. However, for arbitrary $x\in\omega$ and $r\in (0; +\infty)$, if $y\in\omega$ is such that $2^{y}>r$, then $y\notin B_{d^{-1}}(x, r)$. Therefore, $\omega$ is $d^{-1}$-unbounded. 
\end{exam}

\begin{defi} We say that a bornological biuniverse $((X, \tau_1, \tau_2), \mc{B})$ is \textbf{(quasi)-metrizable} if there exists a (quasi)-metric $d$ in $X$ such that $\tau_1=\tau(d), \tau_2=\tau(d^{-1})$ and $\mc{B}=\mc{B}_d(X)$.
\end{defi}

It is obvious that if $\tau$ is a topology in $X$, then a bornological biuniverse $((X, \tau, \tau), \mc{B})$ is metrizable if and only if the bornological universe $((X, \tau),\mc{B})$ is metrizable in the sense of Definition 10.1 of \cite{Hu}. Let us recall this definition.

\begin{defi} Let $((X, \tau),\mc{B})$ be a bornological universe. We say that:
\begin{enumerate}
\item[(i)] $((X, \tau),\mc{B})$ is  \textbf{metrizable (in the sense of Hu)} if there exists a metric $d$ on $X$ such that $\tau=\tau(d)$ and $\mc{B}=\{ A\subseteq X: \text{diam}_d(A)<+\infty\}$ where $\text{diam}_d(A)=\sup\{d(x, y): x, y\in A\}$;
\item[(ii)]  $((X, \tau),\mc{B})$ is \textbf{quasi-metrizable} if there exists a quasi-metric $d$ on $X$ such that  $\tau=\tau(d)$ and, moreover, $\mc{B}$ is the collection of all $d$-bounded sets.
\end{enumerate}
\end{defi}

We show in Section 4 that if a bornological biuniverse $((X, \tau, \tau), \mc{B})$ is quasi-metrizable, then the bornological universe $((X,\tau), \mc{B})$ is metrizable.  

Of course, it is impossible to prove anything in mathematics without axioms. The basic set-theoretic system of axioms used in this paper is \textbf{ZF} (cf. \cite{Ku1}-\cite{Ku2}). If a relatively independent of \textbf{ZF} axiom \textbf{A} is added to \textbf{ZF}, we shall write $\mathbf{ZF+A}$ and clearly denote our theorems proved in $\mathbf{ZF+A}$ but not in $\mathbf{ZF}$. As far as set-theoretic axioms are concerned, we use standard notation from \cite{Ku2} and \cite{Her}. In particular, we denote $\mathbf{ZF+AC}$ by $\mathbf{ZFC}$. If it is necessary, we use a modification of \textbf{ZF} signalled in \cite{PW}. 

According to Theorem 1 of \cite{Vr2} and Theorem 13.2 of \cite{Hu}, the following theorem can be called \textbf{Hu's metrization theorem for bornological universes}:

\begin{thm} 
It holds true in $\mathbf{ZFC}$ that a bornological universe $((X, \tau), \mc{B})$ is metrizable if and only if  it is proper, while, simultaneously, $(X, \tau)$ is metrizable and $\mc{B}$ has a countable base. 
\end{thm}

One of the main aims of our present work is to show that the proof to Hu's metrization theorem in \cite{Hu} highly involves the axiom \textbf{CC} of countable choice and to prove in \textbf{ZF} the following generalization of Theorem 1.9:

\begin{thm} It is true in $\mathbf{ZF}$ that a bornological biuniverse $((X, \tau_1, \tau_2), \mc{B})$ is quasi-metrizable if and only if $\mc{B}$ has a countable base and it is $(\tau_1, \tau_2)$- proper, while the bitopological space $(X, \tau_1, \tau_2)$ is quasi-metrizable.
\end{thm}
  
We deduce Theorem 1.9 from 1.10 and we prove a stronger theorem than 1.10 in Section 4 (Theorem 4.7). We also give some other applications of Theorem 1.10. Especially in Sections 2, 3 and 7, we give examples of unprovable in \textbf{ZF} results on bornological universes that were obtained by other authors probably either in \textbf{ZFC} or in naive preaxiomatic set theory. Section 5 contains a generalization of Theorem 13.5 of \cite{Hu}. We pay a special attention to \cite{GM} and, in Section 6, we modify the basic theorem of \cite{GM} to get necessary and sufficient conditions for a bornological quasi-metric universe to be uniformly quasi-metrizable (Theorem 6.5); furthermore, in Section 8, we modify a theorem about compact bornologies from \cite{GM}. Finally, in Section 10, we offer relevant to bornologies concepts of quasi-metrizability for generalized topological spaces in the sense of Delfs and Knebusch (cf. \cite{DK}, \cite{Pie1}, \cite{Pie2}, \cite{PW}) and give a number of illuminating examples. Section 9 concerns bornologies in generalized topological spaces and it is a preparation  for Section 10. We close the paper with Section 11 where there are remarks about new topological categories.  

We recommend \cite{En} as a monograph on topology that we use. Our basic knowledge about category theory is taken from  \cite{AHS}. Models of set theory applied by us are described in \cite{Her}, \cite{Jech1}-\cite{Jech2} and \cite{HR}.     

\section{Countability}

The axiom of countable choice is usually denoted by \textbf{CC}, \textbf{ACC} or \textbf{CAC}. 

We shall use the following standard notions of finiteness and infinity:
\begin{defi} A set $X$ is called:
\begin{enumerate}
\item[(i)] 
\textbf{finite} or \textbf{T-finite} (truly finite) if there exists $n\in\omega$ such that $X$ is equipollent with $n$;
\item[(ii)] \textbf{D-finite} or \textbf{Dedekind-finite} if no proper subset of $X$ is equipollent with $X$.
\item[(iii)] \textbf{infinite} or \textbf{T-infinite} if it is not finite, and \textbf{ D-infinite} if it is not D-finite.
\end{enumerate}
\end{defi}
 A set is T-finite if and only if it is finite in Tarski's sense (cf. Definition 4.4 of \cite{Her}). Other notions relevant to finiteness were studied, for example, in \cite{Cruz}. The term \emph{truly finite} was suggested by K. Kunen in a private communication with E. Wajch.

Let us establish three distinct notions of countability.
\begin{defi} A set $X$ is called:
\begin{enumerate} 
\item[(i)] \textbf{countable} or \textbf{T-countable} (truly countable) if $X$ is equipollent with a subset of $\omega$;
\item[(ii)] \textbf{D-countable} if every D-infinite subset of $X$ is equipollent with $X$;
\item[(iii)] \textbf{W-countable} if every well-orderable subset of $X$  is D-countable.
\end{enumerate}
\end{defi}
To each notion of countability Q corresponds a notion of uncountability.
\begin{defi}
We say that a set is \textbf{Q-uncountable} if it is not Q-countable where Q stands for T, D or W.
Sets that are T-uncountable are called \textbf{uncountable}.
\end{defi} 

Let us denote by  $\mathbf{CC}(\text{D-fin})$ the following statement: every non-void countable collection of pairwise disjoint non-void D-finite sets has a choice function. As usual, $\mathbf{CC}(\text{fin})$ is the statement: every non-void countable collection of pairwise disjoint non-void finite sets has a choice function.

\begin{prop} The following conditions are equivalent:
\begin{enumerate}
\item[(i)] $\mathbf{CC}(\text{D-fin})$;
\item[(ii)] every D-countable set is countable.
\end{enumerate}
\end{prop}
\begin{proof} Let $X$ be a set. Assume that $X$ is D-countable. If $X$ is D-infinite, then $X$ is countable (cf. \cite{W}, p. 48). Assume that $X$ D-finite. Then if $(i)$ holds, it follows from E13 of Section 4.1 of \cite{Her} that the set $X$ is finite, so countable. Hence $(i)$ implies $(ii)$. Now, assume that $(ii)$ holds and that $X$ is D-finite. Then $X$ is D-countable, so countable. This implies that $X$ is equipollent with a finite subset of $\omega$ and, in consequence, $X$ is finite. By E13 of Section 4.1 of \cite{Her}, $(ii)$  implies $(i)$. 
\end{proof}

\begin{f} For every D-finite set $X$,  the following conditions are equivalent:
\begin{enumerate}
\item[(i)] $X$ is finite;
\item[(ii)] $X\cup\omega$ is D-countable.
\end{enumerate}
\end{f}
\begin{cor} 
If $X$ is an infinite D-finite set, then the set $X\cup\omega$ is D-uncountable.
\end{cor}
\begin{f} A set $X$ is countable if and only if $X\cup\omega$ is D-countable.
\end{f}
\begin{cor} In every model $\mathbf{M}$ for $\mathbf{ZF}$ such that there is in $\mathbf{M}$ an infinite D-finite subset of $\mathbb{R}$, the collection of all D-countable subsets of $\mathbb{R}$ is not a bornology.
\end{cor}
\begin{f} In every set $X$,  the following collections are bornologies:
\begin{enumerate}
\item[(i)] the collection $\mathbf{FB}(X)$ of all finite subsets of $X$;
\item[(ii)] the collection of all D-finite subsets of $X$;
\item[(iii)] the collection of all countable subsets of $X$;
\item[(iv)] the collection of all W-countable subsets of $X$.
\end{enumerate}
\end{f} 

Several remarks on D-countability can be found in \cite{W}.

\section{Second-countable bornological biuniverses}

One may deduce wrongly from Theorem 5.5 of \cite{Hu} that every base of a second-countable boundedness $\mc{B}$ certainly contains a countable base for $\mc{B}$.  However, we are going to prove that Theorem 5.5 of \cite{Hu} is relatively independent of \textbf{ZF}. To do this, let us consider the  following bornologies in $\mb{R}$:
$$\mathbf{UB}(\mb{R})=\{ A\subseteq\mb{R}: \exists_{r\in\mb{R}} A\subseteq (-\infty; r)\},$$
$$\mathbf{LB}(\mb{R})=\{ A\subseteq\mb{R}: \exists_{r\in\mb{R}} A\subseteq (r; +\infty)\}.$$

 Of course, $\mathbf{UB}(\mb{R})$ and $\mathbf{LB}(\mb{R})$ are second-countable. 

\begin{prop} Equivalent are:
\begin{enumerate}
\item[(i)] $\mathbf{CC}(\mb{R})$;
\item[(ii)]  for every unbounded to the right subset $D$ of $\mb{R}$, the collection $\mc{A}(D)=\{( -\infty; d): d\in D\}$ contains a countable base for $\mathbf{UB}(\mb{R})$;
\item[(iii)] for every unbounded to the left subset $D$ of $\mb{R}$, the collection $\mc{A}(D)=\{(d; +\infty): d\in D\}$ contains a countable base for $\mathbf{LB}(\mb{R})$;
\end{enumerate}
\end{prop}
\begin{proof} 
First, assume that $\mathbf{CC}(\mb{R})$ holds and that  $D$ is an unbounded to the right subset of $\mb{R}$. It follows from Theorem 3.8 of \cite{Her} that there exists an unbounded sequence $(d_n)_{n\in\omega}$ of elements of $D\cap[0; +\infty)$. Then $\{(-\infty; d_n): n\in\omega\}$ is a countable base for $\mathbf{UB}(\mb{R})$.  
 
Now, suppose that $\mathbf{CC}(\mathbb{R})$ does not hold. By Theorem 3.8 of \cite{Her}, there exists an unbounded subset $B$ of $\mathbb{R}$ which does not contain any unbounded sequence. Then the set $D=B\cup\{ -x: x\in B\}$ does not contain any unbounded sequence. The collection $\mc{A}(D)$ is a base for $\mathbf{UB}(\mb{R})$ such that $\mc{A}(D)$ does not contain any countable base for $\mathbf{UB}(\mb{R})$. Hence $(i)$ implies $(ii)$.

To show that $(ii)$ and $(iii)$ are equivalent, it suffices to make a suitable use of the mapping $f:\mb{R}\to\mb{R}$ defined by $f(x)=-x$ for $x\in\mb{R}$.  
\end{proof}

\begin{cor} Let $\mathbf{M}$ be any model for $\mathbf{ZF}$ such that $\mathbf{CC}(\mb{R})$ fails in $\mathbf{M}$. Then the bornology $\mathbf{UB}(\mb{R})$ has a base which does not contain any countable base for $\mathbf{UB}(\mb{R})$. In consequence, Theorem 5.5 of \cite{Hu} is false in $\mathbf{M}$. 
\end{cor} 

\begin{prop} ($\mathbf{ZF+CC}$) If a boundedness $\mc{B}$ in $X$ has a countable base, then every base for $\mc{B}$ contains a countable base for $\mc{B}$.
\end{prop}
\begin{proof} 
Let $\mc{A}$ be a base for $\mc{B}$. Consider an arbitrary countable base $\mc{C}$ for $\mc{B}$. Then $\mc{C}\neq\emptyset$. For $C\in \mc{C}$, let $\mc{A}(C)=\{A\in\mc{A}: C\subseteq A\}$. Since $\mc{A}$ is a base for $\mc{B}$, we have $\mc{A}(C)\neq\emptyset$ whenever $C\in\mc{C}$. Using \textbf{CC}, we deduce that there exists $x\in\prod_{C\in \mc{C}}\mc{A}(C)$. Then $\mc{A}_0=\{x(C): C\in \mc{C}\}\subseteq \mc{A}$ and $\mc{A}_0$ is a countable base for $\mc{B}$.
\end{proof}

We can get the following correct modification in $\mathbf{ZF}$ of Theorem 5.5 of \cite{Hu}:

\begin{prop} Let $\mc{C}$ be a countable base for a boundedness $\mc{B}$ in $X$ such that $\mc{B}$ does not have a maximal set with respect to inclusion. Then there exists a strictly increasing sequence $(A_n)$ of members of $\mc{C}$ such that the collection $\{ A_n: n\in\omega\}$ is a base for $\mc{B}$.
\end{prop}
\begin{proof} It follows from the countability of $\mc{C}$ that we can write $\mc{C}=\{ C_n: n\in\omega\}$. Let $A_0=C_0$. Since $\mc{B}$ does not contain maximal bounded sets, there exists $B\in\mc{B}$ such that $B$ is not a subset of $A_0\cup C_1$ and there exists $C\in\mc{C}$ such that $A_0\cup B\cup C_1\subseteq C$. This proves that there exists $C\in\mc{C}$ such that $A_0\cup C_1\neq C$ and $A_0\cup C_1\subseteq C$. Let $n_1=\min\{n\in \omega: A_0\cup C_1\subset C_n\}$ and $A_1= C_{n_1}$. Of course, we use the symbol $\subset$ for strict inclusion. Suppose that, for $m\in\omega\setminus\{0\}$, we have already defined the set $A_m\in\mc{C}$. In much the same way as above, we take $n_{m+1}=\min\{ n\in\omega: A_m\cup C_{m+1}\subset C_{n}\}$ and $A_{m+1}=C_{n_{m+1}}$. The sequence $(A_n)$ has the required properties.
\end{proof}

Although Theorem 5.5 of \cite{Hu} is unprovable in \textbf{ZF}, the following proposition about bornological biuniverses clearly shows that Theorem 5.6 of \cite{Hu} holds true in \textbf{ZF}; however, in the proof of Theorem 5.6 in \cite{Hu}, an illegal in \textbf{ZF} countable choice was involved. Therefore, we offer its  more careful proof in \textbf{ZF}.

\begin{prop} Let us suppose that $(X, \tau_1, \tau_2)$ is a bitopological space, while $\mc{B}$ is a second-countable $(\tau_1, \tau_2)$-proper boundedness in $X$ such that $\mc{B}$ does not have maximal sets with respect to inclusion. Then there exists a strictly increasing sequence $(A_n)$ of $\tau_1$-open sets such that $\mc{A}=\{ A_n: n\in\omega\}$ is a base for $\mc{B}$ such that $\text{cl}_{\tau_2} A_n\subset \text A_{n+1}$ for each $n\in\omega$.
\end{prop}  
\begin{proof} Take, by Proposition 3.4, a strictly increasing countable base $\mc{C}=\{C_n: n\in\omega\}$ for $\mc{B}$.  Let $A_0=\text{int}_{\tau_1} C_0$. Suppose that, for $m\in\omega$, we have already defined a $\tau_1$-open set $A_m\in\mc{B}$. We use similar arguments to the ones given in the proof to Proposition 3.4 with the exception that, since $\mc{B}$ is $(\tau_1, \tau_2)$-proper, we may define $n_{m+1}=\min\{ n\in\omega: \text{cl}_{\tau_2}( A_m\cup C_{m+1})\subset \text{int}_{\tau_1} C_{n}\}$ and $A_{m+1}=\text{int}_{\tau_1} C_{n_{m+1}}$. 
\end{proof} 

\section{Quasi-metrization theorems for bornological biuniverses}

If $\tau$ is a topology on $X$ and if $A\subseteq X$, we denote $\tau\mid  _A=\{ A\cap V: V\in \tau\}$.
For the real line $\mb{R}$, the topology $u=\{\emptyset, \mb{R}\}\cup\{ (-\infty; a): a\in\mb{R}\}$ is called \textbf{the upper topology} on $\mb{R}$, while the topology $l=\{\emptyset, \mb{R}\}\cup\{(a; +\infty): a\in\mb{R}\}$ is called \textbf{the lower topology} on $\mb{R}$ (cf. \cite{FL}, \cite{Sal}). If $A\subseteq \mb{R}$, then we use $(A, u, l)$ as an abbreviation of $(A, u\mid _A, l\mid _A)$ where  $u=u\mid _A$ and $l=l\mid _A$.

\begin{defi} 
Suppose that $(X, \tau^{X}_1, \tau^{X}_2)$ and $(Y, \tau^{Y}_1, \tau^{Y}_2)$ are bitopological spaces. A mapping $f:X\to Y$ is called \textbf{bicontinuous with respect to $(\tau^{X}_1, \tau^{X}_2, \tau^{Y}_1, \tau^{Y}_2)$} (in abbreviation: bicontinuous) if $$\{f^{-1}(V): V\in\tau^{Y}_i\}\subseteq \tau^{X}_i$$ for each $i\in\{1, 2\}$.
\end{defi} 
 
A crucial role in the study of bornologies is played by a concept of a characteristic function of a bornology which is also called a forcing function (cf.\cite{Hu}, \cite{Be}). We need to extend this concept to bornological biuniverses. 
 
\begin{defi} Let $(X, \tau_1, \tau_2)$ be a bitopological space. Then a $(\tau_1, \tau_2)$\textbf{-characteristic function} for a bornology $\mc{B}$ in $X$, is a  bicontinuous function $f:(X, \tau_1, \tau_2)\to ([0; +\infty), u, l)$ such that $$\mc{B}=\{ A\subseteq X: \sup\{ f(x): x\in A\}< +\infty\}.$$
\end{defi}

\begin{f}[cf. 4.1 of \cite{Kel}]
Let $d$ be a quasi-metric on $X$ and let $x_0\in X$. Define $f(x)=d(x_0, x)$ for $x\in X$. Then the function $f:( X, \tau (d), \tau (d^{-1}))\to ([0;+\infty), u, l)$ is bicontinuous. 
\end{f}

\begin{defi} We say that a (quasi)-metric $d$ \textbf{induces a bornological biuniverse} $((X,\tau_1, \tau_2), \mc{B})$ if $\tau_1=\tau(d), \tau_2=\tau(d^{-1})$  and $\mc{B}=\mc{B}_d(X)$. 
\end{defi}

\begin{prop} Suppose that a bornological biuniverse $((X, \tau_1, \tau_2), \mc{B})$ is (quasi)-metrizable. Then there exists a $(\tau_1, \tau_2)$-characteristic function for the bornology $\mc{B}$. 
\end{prop}
\begin{proof} Let us consider an arbitrary point $x_0\in X$ and any (quasi)-metric $d$ such that $d$ induces the bornological biuniverse $((X, \tau_1, \tau_2), \mc{B})$. Then, by Fact 4.3, a $(\tau_1, \tau_2)$-characteristic function for $\mc{B}$ is the function $f: X\to \mb{R}$ where $f(x)=d(x_0, x)$ for $x\in X$. 
\end{proof}

\begin{prop} Suppose that a bornological biuniverse $((X, \tau_1, \tau_2), \mc{B})$ is such that $\mc{B}$ has a $(\tau_1, \tau_2)$-characteristic function. Then $\mc{B}$ is both second-countable and $(\tau_1, \tau_2)$-proper. 
\end{prop}
\begin{proof} 
Let $f$ be any $(\tau_1, \tau_2)$-characteristic function for $\mc{B}$. For $n\in\omega$, let $A_n=f^{-1}((-\infty, n])$. Then the collection $\{A_n: n\in\omega\}$ is a countable base for $\mc{B}$ such that $\text{cl}_{\tau_2} A_{n}\subseteq \text{int}_{\tau_1}A_{n+1}$.
\end{proof}

\begin{thm} Let us suppose that $(X, \tau_1, \tau_2)$ is a (quasi)-metrizable bitopological space and that $\mc{B}$ is a bornology in $X$. Then the following conditions are all equivalent:
\begin{enumerate}
\item[(i)] the bornological biuniverse $((X, \tau_1, \tau_2), \mc{B})$ is (quasi)-metrizable;
\item[(ii)] there exists a $(\tau_1,\tau_2)$-characteristic function for $\mc{B}$;
\item[(iii)] the bornology $\mc{B}$ is $(\tau_1, \tau_2)$-proper and it has a countable base.
\end{enumerate}
\end{thm}
\begin{proof}Let us consider any quasi-metric $\sigma$ on $X$ such that $\tau_1=\tau (\sigma)$ and $\tau_2=\tau (\sigma^{-1})$. Put   $d(x,y)=\min\{\sigma(x,y), 1\}$ for $x,y\in X$. It is easy to observe that if $X\in\mc{B}$, then all conditions $(i)-(iii)$ are fulfilled. Assume that $X\notin\mc{B}$. It follows from Proposition 4.5  that $(i)$ implies $(ii)$. Assume $(ii)$ and suppose that $f$ is a $(\tau_1, \tau_2)$-characteristic function for $\mc{B}$. For $x,y\in X$, let $\rho (x, y)=d(x, y)+\max\{f(y)-f(x), 0\}$. Then the quasi-metric $\rho$ induces the bornological biuniverse $((X, \tau_1, \tau_2), \mc{B})$. In the case when $\sigma$ is a metric, we can put $\rho (x, y)=d(x, y)+\mid f(y)-f(x)\mid $ to obtain a metric that induces $((X, \tau_1, \tau_2), \mc{B})$. Hence $(ii)$ implies $(i)$. 

Now, assume that $(iii)$ holds. Since $X\notin\mc{B}$, it follows from Proposition 3.5 that there exists a base $\{A_n: n\in\omega\}$ for $\mc{B}$ such that $\text{cl}_{\tau_2}A_n$ is a proper subset of $ \text{int}_{\tau_1}A_{n+1}$ for each $n\in\omega$. We may assume that $A_0=\emptyset$. For $n\in\omega\setminus\{0\}$ and $x\in X$, let $f_n(x)=d(\text{cl}_{\tau_2}A_n, x)$ and $g_n(x)=d(x, X\setminus\text{int}_{\tau_1}A_{n+1})$. Then $f_n:(X,\tau_1, \tau_2)\to([0; +\infty), u, l)$ and $g_n:(X,\tau_1, \tau_2)\to ([0; +\infty), l, u)$ are bicontinuous. For each $x\in X$, we have  $f_n(x)+g_n(x)\neq 0$, so we can put $h_n(x)=\frac{f_n(x)}{f_n(x)+g_n(x)}$. Moreover, we define $h_0(x)=1$ for each $x\in X$. It is easy to check that the function $h_n:(X, \tau_1,\tau_2)\to ([0; 1], u, l)$ is bicontinuous for each $n\in\omega$ (cf. the proof to Corollary 2.2.16 in \cite{Sal}). Let $\psi(x)=h_n(x)+n$ when $x\in\text{int}_{\tau_1}A_{n+1}\setminus\text{int}_{\tau_1}A_n$. We are going to prove that the function $\psi$ is bicontinuous with respect to $(\tau_1, \tau_2 , u, l)$. 

Let $x\in\text{int}_{\tau_1}A_{n+1}\setminus\text{int}_{\tau_1}A_n$ and $y\in\text{int}_{\tau_1}A_{m+1}\setminus\text{int}_{\tau_1}A_m$. Consider any real numbers $r, s$ such that $r< \psi(x)< s$. We assume that $n\neq 0$.  There exists $U_s\in\tau_1$ such that $x\in U_s\subseteq\text{int}_{\tau_1}A_{n+1}$ and if $y\in U_s$, then $h_n(y)+n< s$. There exists $V_r\in\tau_2$ such that $x\in V_r\subseteq X\setminus\text{cl}_{\tau_2} A_{n-1}$ and if $y\in V_r$,  then $h_n(y)+n>r$. Of course, if $m=n$, then $\psi(y)<s$ when $y\in U_s$, while $\psi(y)>r$ when $y\in V_r$. Let us assume that $m\neq n$.
Suppose that $y\in U_s$. Then $m<n$, so $\psi(y)\leq 1+m\leq n\leq \psi(x)<s$. 

Suppose that $y\in V_r$. If $m>n$, we have $\psi(y)\ge m\ge 1+n\ge\psi(x)>r$. Let $m<n$. Since $y\notin\text{int}_{\tau_1}A_{n-1}$, we have $m+1\ge n$.  As $m+1\leq n$, we have $m+1=n$. If $x\notin\text{cl}_{\tau_2}A_n$ we could take $V_r^{*}=V_r\cap( X\setminus\text{cl}_{\tau_2}A_n)\in\tau_2$ and observe that if $y\in V^{*}_r$, then $m\ge n$ and $\psi(y)>r$. Let us consider the case when $m<n$ and $x\in\text{cl}_{\tau_2}A_n$. Then  $\psi(x)=n$. We take a positive real number $\epsilon$ such that $n-\epsilon>r$. Since $h_{n-1}(x)=1$, there exists $W_{\epsilon}\in\tau_2$ such that $x\in W_{\epsilon}$ and $h_{n-1}(t)>1-\epsilon$ for each $t\in W_{\epsilon}$. If $y\in W_{\epsilon}\cap V_r$ and $m+1=n$, then $\psi(y)=h_{n-1}(y)+n-1> 1-\epsilon +n-1>r$. The case when $n=0$ is also obvious now.  This completes the proof that $\psi$ is bicontinuous with respect to $(\tau_1, \tau_2, u, l)$. It is easy to check that $\mc{B}=\{ A\subseteq X: \sup\psi(A)<+\infty\}$, so  $\psi$ is a $(\tau_1, \tau_2)$-characteristic function for $\mc{B}$. Hence $(ii)$ follows from $(iii)$. To complete the proof, it suffices to apply Proposition 4.6.  
\end{proof}

\begin{cor} Theorem 1.10 is true. 
\end{cor}

\begin{cor} The assumption of $\mathbf{ZFC}$ can be weakened to $\mathbf{ZF}$ in Theorem 1.9.
\end{cor}
 
 \begin{cor} Let us suppose that $(X, \tau)$ is a topological space and $\mc{B}$ is a bornology in $X$. Then the bornological biuniverse $((X, \tau, \tau), \mc{B})$ is quasi-metrizable if and only if the bornological universe $((X, \tau), \mc{B})$ is metrizable.
 \end{cor}
 \begin{proof} It suffices to prove that if there exists a quasi-metric which induces $((X, \tau, \tau), \mc{B})$, then $((X, \tau), \mc{B})$ is metrizable. Let $d$ be a quasi-metric in $X$ such that $\tau=\tau(d)=\tau(d^{-1})$ and $\mc{B}=\mc{B}_d(X)$. Define $\rho=\max\{ d, d^{-1}\}$. Then $\rho$ is a metric in $X$ such that $\tau(\rho)=\tau$. Moreover, by Theorem 4.7, the bornology $\mc{B}$ is second-countable and $\tau$-proper; hence, the bornological universe $((X, \tau), \mc{B})$ is metrizable by Theorem 4.7. 
 \end{proof} 
 
 \begin{exam} Let $d$ be the quasi-metric from Example 1.6. Then $\tau(d)=\tau(d^{-1})=\mc{P}(\omega)$. Moreover, $\mc{B}_d(\omega)=\mc{P}(\omega)$ and $\mc{B}_{d^{-1}}(\omega)=\mathbf{FB}(\omega)$. The metric $\rho=\max\{d, d^{-1}\}$ does not induce $((\omega, \mc{P}(\omega), \mc{P}(\omega)), \mc{B}_{d}(\omega))$; however, $\rho$ induces $((\omega, \mc{P}(\omega), \mc{P}(\omega)), \mathbf{FB}(\omega))$.
 \end{exam}

 \begin{exam} Let $\tau_{S,r}$ be the right half-open interval topology in $\mb{R}$ and let $\tau_{S, l}$ be the left half-open interval topology in $\mb{R}$. Then $(\mb{R}, \tau_{S, r})$ is the Sorgenfrey line. 
 \begin{enumerate}
  \item[(i)]  The bornological biuniverse $((\mb{R}, \tau_{S,r}, \tau_{S,l}), \mathbf{UB}(\mb{R}))$
 is not metrizable but it is quasi-metrizable by the following quasi-metric $\rho_{S}$:
  $$ \rho_{S}(x,y)=\left\{ \begin{array}{ll}
 y-x, & x\le y\\
 1, & x>y.\end{array}\right.$$
 Let us notice that  $\mc{B}_{\rho^{-1}_{S}}(\mb{R})=\mathbf{LB}(\mb{R})$ and the quasi-metric $\rho_S$ does not induce the bornological biuniverse $((\mb{R}, \tau_{S,r}, \tau_{S,l}), \mathbf{LB}(\mb{R}))$. However, the bornological biuniverse $((\mb{R}, \tau_{S,r}, \tau_{S,l}), \mathbf{LB}(\mb{R}))$ is induced by the quasi-metric $\rho_{L}$ defined as follows:
  $$ \rho_{L}(x,y)=\left\{ \begin{array}{ll}
 \min\{y-x, 1\}, & x\le y\\
 1+x-y, & x>y.\end{array}\right.$$
 
 \item[(ii)] The non-metrizable bornological biuniverse $((\mb{R}, \tau_{S,r}, \tau_{S,l}), \mc{P}(\mb{R}))$ is quasi-metrizable by the quasi-metric $\rho_{S,1}$ defined as follows:
   $$ \rho_{S,1}(x,y)=\left\{ \begin{array}{ll}
 \min\{1,y-x\}, & x\le y\\
 1, & x>y.\end{array}\right.$$
 \end{enumerate}
 \end{exam}

\begin{exam} We consider the following \textbf{hedgehog-like scheme}. Let $(X, d)$ be a quasi-metric space such that $X$ has at least two distinct points. Let $S$ be a non-empty set. We fix $x_0\in X$ and put  $Y_s= (X\setminus \{x _0\})\times\{ s\}$ for $s\in S$. Let us fix $p\notin\bigcup_{s\in S} Y_s$ and put $Y=\{p\}\cup\bigcup_{s\in S}Y_s$. Let $x, y\in X\setminus\{x_0\}$ and $s, s'\in S$. We define $\rho(p, p)=0, \rho( (x, s), p)= d(x, x_0), \rho(p, (x, s))= d(x_0, x)$ and $\rho((x, s), (y, s))=d (x, y)$. If $s\neq s'$, we put $\rho((x, s), (y, s'))=d(x, x_0)+ d(x_0, y)$. Let us consider the collection $\mc{B}$ of all sets $A\subseteq Y$ such that there are finite $S(A)\subseteq S$ such that $A\subseteq \{p\}\cup\bigcup_{s\in S(A)}Y_s$. Then $\mc{B}$ is a bornology in $Y$. If $S$ is countable, then $\mc{B}$ is second-countable. If $S$ is infinite and, simultaneously,  $x_0$ is an accumulation point of $(X, \tau(d))$, then the bornology $\mc{B}$ is not $(\tau(\rho), \tau(\rho^{-1}))$-proper. Let us denote the bornological biuniverse $((Y, \tau(\rho), \tau(\rho^{-1})), \mc{B})$ by $J(X, d, x_0, S)$ and let $Y(X, d, x_0, S)=(Y, \tau(\rho))$. We can apply $J(X, d, x_0, S)$ as follows.
\begin{enumerate}
\item[(i)] If $X=[0; 1]$ and $d(x, y)=\mid x-y\mid $ for $x, y\in X$, then the bornological universe $J(X, d, 0, \omega)$ is not quasi-metrizable although its bornology is second-countable. In this case, $Y(X, d, x_0, \omega)$ is the hedgehog space of spininess $\omega$ (cf. 4.1.5 of \cite{En}), so we can call $((Y, \tau(\rho)),\mc{B})$ \textbf{the bornological hedgehog space of spininess $\omega$}. 
\item[(ii)] If $\rho_S$ is the quasi-metric defined in Example 4.12 (i), then the bornological biuniverse $J(\mb{R}, \rho_S, 0, \omega )$ is not quasi-metrizable but its bornology has a countable base. 
\item[(iii)]Let $C$ be the unit circle in $\mb{R}^2$. We fix $x_0\in C$ and we consider the Euclidean metric $d_e$ in $C$.  The bornological biuniverse $J(C, d_e, x_0, \omega)$ is not quasi-metrizable although its bornology has a countable base. We can call $J(C, d_e, x_0, \omega)$ \textbf{the bornological metric wedge sum of circles}. In this case, the topological space $Y(C, d_e, x_0, \omega)$ is not compact.  
\item[(iv)] It is worthwhile to compare $J(C, d_e, x_0, \omega)$ with \textbf{the bornological Hawaiian earring} $(H, \mc{B}_H)$ where  $H=\bigcup_{n\in\omega\setminus\{0\}} H_n$ is considered with its natural topology inherited from $\mb{R}^2$ and, for each $n\in\omega\setminus\{0\}$, the set $H_n$ is the circle with centre $(\frac{1}{n}, 0)$ and radius $\frac{1}{n}$, while $\mc{B}_H$ is the collection of all sets $A\subseteq H$ such that there exist sets $n(A)\in\omega$ such that $A\subseteq\bigcup_{n\in n(A)\setminus\{0\}}H_n$. Then $H$ is compact and the bornology $\mc{B}_H$ has a countable base. Since there does not exist $A\in\mc{B}_H$ such that $(0,0)\in\text{int}_{d_e}A$, it follows from Theorem 4.7 that the bornological universe $(H, \mc{B}_H)$ is not quasi-metrizable. 
\end{enumerate}
\end{exam}
 
 In view of the examples above, when $d$ is a quasi-metric in $X$ and $\mc{B}$ is a bornology in $X$ but $d$ does not induce the bornological biuniverse $(X, \mc{B})=((X, \tau(d), \tau(d^{-1})), \mc{B})$,  it might be interesting to find, in terms of $d$, necessary and sufficient conditions for $(X, \mc{B})$ to be quasi-metrizable. To do this, we need the following concept:
 
 \begin{defi}
Let $d$ be a quasi-pseudometric in a set $X$ and let $\delta\in (0; +\infty)$. For a set $A\subseteq X$, \textbf{the $\delta$-neighbourhood of $A$ with respect to $d$} is  the set $[A]^{\delta}_d=\bigcup_{a\in A}B_d(a,\delta)$.
\end{defi}
 Let us notice that if $\emptyset\neq A\subseteq X$, then $[A]^{\delta}_d=\{ x\in X: d(A, x)<\delta\}$.
 
 \begin{thm} For every bornological biuniverse $((X, \tau_1, \tau_2), \mc{B})$, the following conditions are equivalent:
\begin{enumerate}
\item[(i)]  $((X, \tau_1, \tau_2), \mc{B})$ is (quasi)-metrizable;
\item[(ii)] there exists a (quasi)-metric $d$ in $X$ such that $\tau_1=\tau(d), \tau_2=\tau(d^{-1})$ and $\mc{B}$ has a base $\{ B_n: n\in\omega\}$ with the following property:
$$\forall_{n\in\omega}\exists_{\delta\in(0; +\infty)}[B_n]^{\delta}_d\subseteq B_{n+1}.$$
\end{enumerate} 
\end{thm}

\begin{proof} Let $((X, \tau_1, \tau_2), \mc{B})$ be a bornological biuniverse. Suppose that $(i)$ holds and that $d$ is a (quasi)-metric in $X$ such that $d$ induces $((X, \tau_1, \tau_2), \mc{B})$. We consider an arbitrary $x_0\in X$ and, for $n\in\omega$, we define $B_n=B_d(x_0, n+1)$. Since $[B_n]^{\frac{1}{2}}_d\subseteq B_{n+1}$, we infer that $(ii)$ follows from $(i)$. 

Assume that $(ii)$ is satisfied. Let $C\subseteq X$ and $D\subseteq X$ be such that, for some $\delta\in (0; +\infty)$, the inclusion $[C]^{\delta}_d\subseteq D$ holds. Let $x\in \text{cl}_{\tau_2}C$. There exists $y\in C\cap B_{d^{-1}}(x, \delta)$. Then $d(y, x)<\delta$, so $x\in [C]^{\delta}_d$. Therefore $\text{cl}_{\tau_2}C\subseteq [C]^{\delta}_d$. Of course, since $[C]^{\delta}_d\subseteq D$, we have $[C]^{\delta}_d\subseteq \text{int}_{\tau_1}D$. In consequence, $\text{cl}_{\tau_2}C\subseteq \text{int}_{\tau_1}D$. Now, we deduce from Theorem  4.7 that $(ii)$ implies $(i)$.
\end{proof}

\begin{cor} For every bornological universe $((X, \tau), \mc{B})$, the following conditions are equivalent:
\begin{enumerate}
\item[(i)]  the universe $((X, \tau), \mc{B})$ is (quasi)-metrizable;
\item[(ii)] there exists a (quasi)-metric $d$ in $X$ such that $\tau=\tau(d)$ and, simultaneously,  $\mc{B}$  has a base $\{ B_n: n\in\omega\}$ with the following property:
$$\forall_{n\in\omega}\exists_{\delta\in(0; +\infty)}[B_n]^{\delta}_d\subseteq B_{n+1};$$
\end{enumerate}
\end{cor}

The following example shows that the sets $B_n$ can be $d$-unbounded in Theorem 4.15 and Corollary 4.16.

\begin{exam} For the bornological biuniverse $((\mb{R}, \tau_{S,r}, \tau_{S,l}), \mathbf{LB}(\mb{R}))$ and for the quasi-metric $\rho_S$
  from Example 4.12, the sets $B_n=[-n; +\infty)$ with $n\in\omega$ satisfy condition $(ii)$ of Theorem 4.15 if we put $d=\rho_{S}$ and $\delta=1$. However, the sets $B_n$ are all $\rho_S$-unbounded.  
\end{exam}
 
We offer a number of other relevant examples in Section 10. 

\section{The kernel of a boundedness}

If $X$ is a topological space and $\mc{B}$ is a boundedness in $X$, a notion of a kernel of the universe $(X, \mc{B})$ was  introduced in Definition 6.3 in \cite{Hu}. We adapt this notion to our needs.
 
\begin{defi} Let $\tau$ be a topology in a set $X$. If $\mc{B}$ is a boundedness in $X$, then the $\tau$\textbf{-kernel} of $\mc{B}$ is the set $${\Lambda}_{\tau}(\mc{B})=\bigcup\{ \text{int}_{\tau}A: A\in\mc{B}\}.$$
\end{defi}

\begin{defi}
Let $ (X, \tau_1, \tau_2)$ be a bitopological space. Suppose that $\mc{B}$ is a boundedness in $X$ and put $\Lambda= {\Lambda}_{\tau_1}(\mc{B})$. Let  $\mc{B}_{\Lambda}=\{ A\cap\Lambda: A\in \mc{B}\}$. Then the ordered pair $((\Lambda, \tau_1\mid _{\Lambda}, \tau_2\mid _{\Lambda}), \mc{B}_{\Lambda})$ will be called \textbf{the bornological biuniverse induced by} $\mc{B}$. 
\end{defi}

\begin{f} Suppose that $ (X, \tau_1, \tau_2)$ is a bitopological space. If $\mc{B}$ is a $(\tau_1, \tau_2)$-proper boundedness in $X$ and if $\Lambda= {\Lambda}_{\tau_1}(\mc{B})$, then the bornology $\mc{B}_{\Lambda}$ in $\Lambda$ is $(\tau_1\mid _{\Lambda}, \tau_2\mid _{\Lambda})$-proper.
\end{f}
 
For a topological space $X=(X, \tau)$ and a boundedness $\mc{B}$ in $X$, when $\Lambda= {\Lambda}_{\tau}(\mc{B})$,  Theorem 13.5 of \cite{Hu} concerns the problem of the metrizability of the bornological universe $(\Lambda, \mc{B}_{\Lambda})$ under the assumption that $\Lambda$ is a separable metrizable subspace of $X$. However, the proof to Theorem 13.5 in \cite{Hu} is not in $\mathbf{ZF}$. We give a generalization to bornological universes of Theorem 13.5 of \cite{Hu} and show its proof in $\mathbf{ZF}$. We also show that the assumption of separability is needless in Theorem 13.5 of \cite{Hu}.

\begin{thm} Assume that $(X, \tau_1, \tau_2)$ is a bitopological space and that $\mc{B}$ is a second-countable  $(\tau_1, \tau_2)$-proper boundedness in $X$. Let $\Lambda$ be the $\tau_1$-kernel of $\mc{B}$ and suppose that the bitopological space $(\Lambda, \tau_1\mid _{\Lambda}, \tau_2\mid _{\Lambda})$ is quasi-metrizable. Then there exists a quasi-metric $\rho$ on $\Lambda$ such that the following conditions are satisfied:
\begin{enumerate}
\item[(i)] $\tau_1\mid_{\Lambda}=\tau(\rho)$ and $\tau_2\mid_{\Lambda}=\tau(\rho^{-1})$;
\item[(ii)] $\mc{B}$ is the collection of all $\rho$-bounded subsets of $\Lambda$;
\item[(iii)] for each pair of points $x_0\in\Lambda$, $x_{\ast}\in X\setminus\Lambda$ and for each positive real number $b$, there exists $G\in\tau_2$ such that $x_{\ast}\in G$ and $\rho(x_0, x)>b$ whenever $x\in G\cap\Lambda$.
\end{enumerate} 
\end{thm}

\begin{proof} 
Since $\mc{B}$ is $(\tau_1, \tau_2)$-proper, we have $\mc{B}=\mc{B}_{\Lambda}$. In view of Fact 5.3, $\mc{B}$ is $(\tau_1\mid_{\Lambda}, \tau_2\mid_{\Lambda})$-proper. In the light of Theorem 4.7, there exists in $\mathbf{ZF}$ a quasi-metric $\rho$ in $\Lambda$ such that both conditions $(i)$ and $(ii)$ are satisfied. Let $x_0\in\Lambda$ and $x_{\ast}\in X\setminus\Lambda$. Consider an arbitrary positive real number $b$. Put $B=\{ x\in \Lambda: \rho(x_0, x)\leq b\}$. Of course, $B\in\mc{B}$. Since $\mc{B}$ is $(\tau_1, \tau_2)$-proper, there exists $U\in\tau_1$ such that $B\subseteq U$. Using the assumption that $\mc{B}$ is $(\tau_1, \tau_2)$-proper once again, we deduce that $\text{cl}_{\tau_2}U\subseteq\Lambda$. Let $G=X\setminus\text{cl}_{\tau_2}U$. Then $G\in\tau_2$, $G\cap\Lambda\subseteq\Lambda\setminus B$ and $x_{\ast}\in G$. It is evident that if $x\in G\cap\Lambda$, then $\rho(x_0, x)>b$.  
\end{proof}
 Now, we can immediately deduce in $\mathbf{ZF}$ the following improvement of Theorem 13.5 of \cite{Hu}:
\begin{cor} 
If $\mc{B}$ is a second-countable proper boundedness in a topological space $X$ such that the set $\Lambda=\bigcup\mc{B}$ is a metrizable subspace of $X$, then there exists a metric $\rho$ on $\Lambda$ such that the following conditions are satisfied:
\begin{enumerate}
\item[(i)] the topology of $\Lambda$ as a subspace of $X$ is induced by $\rho$;
\item[(ii)]  $\mc{B}=\{ A\subseteq \Lambda: \text{diam}_{\rho}(A)<+\infty\}$;
\item[(iii)]  for each pair of points $x_0\in\Lambda$, $x_{\ast}\in X\setminus\Lambda$ and for each positive real number $b$, there exists an open set $G$ in $X$ such that $x_{\ast}\in G$ and $\rho (x_0, x)>b$ whenever $x\in G\cap\Lambda$. 
\end{enumerate}
\end{cor}

\section{Uniformly quasi-metrizable bornologies}

This section has been inspired by the necessary and sufficient conditions for a bornology to be uniformly metrizable given in \cite{GM}. We adapt the conditions of Theorem 2.4 of \cite{GM} to bornologies in quasi-metric spaces.

For $x, y\in \mathbb{R}$, let us put $$\rho_{u}(x, y)=\max\{y-x, 0\},  \rho_{l}(x, y)=\max\{x-y,0\}.$$ Then $\rho_u, \rho_l$ are quasi-pseudometrics in $\mathbb{R}$ such that $\rho_{u}^{-1}=\rho_{l}$; moreover, $\tau(\rho_u)$  is the upper topology $u$ in $\mathbb{R}$, while $\tau(\rho_l)$ is the lower topology $l$ in $\mathbb{R}$.

\begin{defi} Let $d_X, d_Y$ be quasi-pseudometrics in sets $X$ and $Y$, respectively. We say that a mapping $f: X\to Y$ is \textbf{$(d_X, d_Y)$-uniformly continuous} if the following condition is satisfied:
 $$\forall_{\epsilon\in (0; +\infty)}\exists _{\delta\in (0;+\infty)}\forall_{x_1, x_2\in X}[d_X(x_1, x_2)<\delta\Rightarrow d_Y(f(x_1), f(x_2))<\epsilon].$$
\end{defi}
 
 \begin{defi}
 Quasi-pseudometrics $d_0, d_1$ in a set $X$ are called \textbf{uniformly equivalent} if the following condition holds:
 $$\forall_{\epsilon\in (0; +\infty)}\exists_{\delta_0, \delta_1\in (0;+\infty)}\forall_{x, y\in X}\forall_{i\in\{0, 1\}}[d_i(x, y)<\delta_i\Rightarrow d_{1-i}(x, y)<\epsilon].$$
 \end{defi}
 
\begin{defi}
Suppose that $(X, d)$ is a quasi-metric space and that $\mc{B}$ is a bornology in $X$. We say that $\mc{B}$ is \textbf{uniformly quasi-metrizable with respect to $d$} if there exists a quasi-metric $\rho$ in $X$ such that $d$ and $\rho$ are uniformly equivalent, while $\mc{B}$ is the collection of all $\rho$-bounded sets.
\end{defi}

\begin{defi}
 We say that quasi-metrics $d, \rho$ in $X$ are \textbf{uniformly locally identical} if they are uniformly equivalent and there exists $\delta\in (0; +\infty)$ such that, for all $x, y\in X$, we have $\rho(x, y)=d(x, y)$ whenever $d(x, y)<\delta$ (cf. \cite{WJ} and Remark 2.5 of \cite{GM}).
\end{defi}

\begin{thm} 
Suppose that $(X, d)$ is a quasi-metric space and that $\mc{B}$ is a bornology in $X$. Then the following conditions are all equivalent:
\begin{enumerate}
\item[(i)] $\mc{B}$ is uniformly quasi-metrizable with respect to $d$;
\item[(ii)] $\mc{B}$ has a base $\{ B_n: n\in\omega\}$ such that, for some $\delta\in (0; +\infty)$ and for each $n\in\omega$, the inclusion $[B_n]^{\delta}_d\subseteq B_{n+1}$ holds;
\item[(iii)] there exists a quasi-metric $\rho$ in $X$ such that $d, \rho$ are uniformly locally identical and $\mc{B}$ is the collection of all $\rho$-bounded sets.
\item[(iv)] there exists a $(d, \rho_u)$-uniformly continuous $(\tau(d), \tau(d^{-1}))$-characteristic function for $\mc{B}$;
\end{enumerate}
\end{thm}
\begin{proof} 
Assume $(i)$ and suppose that $\rho$ is a uniformly equivalent with $d$ quasi-metric in $X$ such that $\mc{B}$ is the collection of all $\rho$-bounded sets. Let $x_0\in X$ and, for $n\in\omega $, let $B_n=B_{\rho}(x_0, n+1)$.  We choose  $\delta\in (0; +\infty)$  such that $\rho( x, y)<\frac{1}{2}$ whenever $d(x, y)<\delta$. Then $[B_n]^{\delta}_d\subseteq [B_n]^{\frac{1}{2}}_{\rho}\subseteq B_{n+1}$ for each $n\in\omega$, so $(ii)$ follows from $(i)$. 

Now, let us suppose that $(ii)$ holds. We may assume that $\delta\in (0; +\infty)$ and that $\{ B_n: n\in\omega\}$ is a base for $\mc{B}$ such that $B_0=\emptyset, B_1\neq\emptyset$ and $[B_n]^{\delta}_d\subseteq B_{n+1}$ for each $n\in\omega$. We shall mimic the proof to Proposition 2.2 in \cite{GM} and change parts of it to show that $(iii)$ follows from $(ii)$. We define $\phi_0(x)=1$ for each $x\in X$. If $n\in\omega\setminus\{0\}$, we define $\phi_n(x)=\min\{1, \frac{1}{\delta}d(B_n, x)\}$ for each $x\in X$. It is easy to check that the function $\phi_n$ is $(d, \rho_u)$-uniformly continuous; moreover, $\phi_n(B_n)\subseteq \{0\}$ and $\phi_n(X\setminus B_{n+1})\subseteq\{1\}$. Let us consider the function $\chi: X\to [0; +\infty)$ defined by
$$ \chi(x)=n-2+\phi_{n-1}(x)$$
for each $x\in B_n\setminus B_{n-1}$ and for each $n\in\omega\setminus\{0\}$. To prove that $\chi$ is $(d, \rho_u)$-uniformly continuous, let us consider an arbitrary pair $x, y$ of points of $X$ such that $d(x, y)<\delta$. Let $n\in\omega$ be the unique natural number such that $x\in B_n\setminus B_{n-1}$. If $z\in X\setminus B_{n+1}$, then $d(x, z)\ge\delta$. This implies that $y\in B_{n+1}$. Let $m\in\omega\setminus\{0\}$ be the unique natural number such that $y\in B_m\setminus B_{m-1}$. Then $m\leq n+1$. We have $\chi(y)-\chi(x)=m-n+\phi_{m-1}(y)-\phi_{n-1}(x)$. If $m=n+1$, then $\chi(y)-\chi(x)=1+\phi_{n}(y)-\phi_{n-1}(x)=\phi_n(y)-\phi_n(x)+\phi_{n-1}(y)-\phi_{n-1}(x)\leq\frac{2}{\delta}d(x, y)$. If $m=n$, then $\chi(y)-\chi(x)=\phi_{n-1}(y)-\phi_{n-1}(x)\leq\frac{1}{\delta}d(x, y)$. Suppose that $m<n$. Then $m-n+1\leq 0, x\notin B_m, y\in B_{n-1}$ and $\chi(y)-\chi(x)=m-n+1+\phi_{m-1}(y)-\phi_{m-1}(x)+\phi_{n-1}(y)-\phi_{n-1}(x)\leq \frac{2}{\delta}d(x, y)$. In consequence, $\chi$ is $(d, \rho_u)$-uniformly continuous. Therefore, $\chi: (X, \tau(d), \tau(d^{-1}))\to (\mathbb{R}, u, l)$ is bicontinuous. Since, for $A\subseteq X$, we have that $A\in\mc{B}$ if and only if $\chi$ is bounded on $A$, the function $\chi$ is a $(\tau(d), \tau(d^{-1}))$-characteristic function for $\mc{B}$.  In much the same way as in Remark 2.5 of \cite{GM}, we can define $\rho(x, y)=\max\{ \min\{ d(x, y), 1\}, \frac{\delta}{2}\max\{\chi(y)-\chi(x), 0\}\}$ to get a quasi-metric $\rho$ uniformly locally identical with $d$ such that $\mc{B}$ is the collection of all $\rho$-bounded sets.  Thus $(ii)$ implies $(iii)$.

Let us assume that $(iii)$ holds. We take a quasi-metric  $\rho$ in $X$ such that $d$ and $\rho$ are uniformly locally identical and $\mc{B}=\mc{B}_{\rho}(X)$.  We fix $x_0\in X$ and define $f(x)=\rho(x_0, x)$ for $x\in X$ to get a $(\tau(\rho), \tau(\rho^{-1}))$-characteristic function $f$ for $\mc{B}$ such that $f$ is $(d, \rho_u)$-uniformly continuous. Hence $(iii)$ implies $(iv)$. 

Finally, we suppose that $(iv)$ holds and we consider an arbitrary function $g$ such that $g$ is a $(\tau(d), \tau(d^{-1}))$-characteristic function for $\mc{B}$ and  $g$ is $(d, \rho_u)$-uniformly continuous. For $x, y\in X$, we can define $d_g(x, y)=\min\{ d(x, y), 1\}+\max\{ g(y)-g(x), 0\}$ to see that $(iv)$ implies $(i)$. 
\end{proof}

\begin{cor} Theorem 2.4 and Remark 2.5 of \textrm{\cite{GM}} hold true in $\mathbf{ZF}$.
\end{cor}

One can use Example 10.16 (i)-(iii) given at the end of Section 10 to see that, for a quasi-metric $d$ in $X$ and a bornology $\mc{B}$ in $X$, it may happen that the bornological universe $((X, \tau(d)), \mc{B})$ is quasi-metrizable or even metrizable, while $\mc{B}$ is not uniformly quasi-metrizable with respect to $d$. 

\section{Applications to independence from ZF}

Mimicking \cite{GM}, let us consider the following bornologies in a metric space $(X, d)$: the bornology $\mathbf{FB}(X)$ of all finite subsets of $X$, the bornology $\mathbf{CB}_d(X)$ generated by the compact subsets of $(X, d)$, the bornology $\mathbf{TB}_d(X)$ of all totally bounded subspaces of $(X, d)$, as well as the bornology $\mathbf{BB}_d(X)$ of all Bourbaki-bounded sets. Several theorems about equivalents of the uniform metrizability of the bornologies $\mathbf{FB}(X)$, $\mathbf{CB}_d(X)$, $\mathbf{TB}_d(X)$  and $\mathbf{BB}_d(X)$ in $\mathbf{ZFC}$ were proved in \cite{GM}. We are going to show that some of the above-mentioned theorems of \cite{GM}  are independent of $\mathbf{ZF}$, while other theorems of \cite{GM} can be proved in $\mathbf{ZF}$. Clearly, we have already shown in the previous section that both Proposition 2.2 and Theorem 2.4 of \cite{GM} hold true in $\mathbf{ZF}$. 

The following theorem will be helpful: 

\begin{thm}
Equivalent are:
\begin{enumerate}
\item[(i)] $\mathbf{CC}(\text{fin})$;
\item[(ii)] for every discrete space $X$, the bornological universe $(X, \mathbf{FB}(X))$ is metrizable (in the sense of Hu) if and only if $X$ is countable.
\end{enumerate}
\end{thm}
\begin{proof}
Assume $(i)$ and let $X$ be any discrete space such that the bornological universe $(X, \mathbf{FB}(X))$ is metrizable.  It follows from Theorem 4.7 that $\mathbf{FB}(X)$ has a countable base. If $\mc{A}$ is a countable base for $\mathbf{FB}(X)$, then $X=\bigcup\mc{A}$, so, by Proposition 3.5 of \cite{Her},  $X$ is countable if $\mathbf{CC}(\text{fin})$ holds. It is obvious that if $X$ is a countable discrete space, then the bornological universe $(X, \mathbf{FB}(X))$ is metrizable in $\mathbf{ZF}$ by Theorem 4.7

Now, assume that $\mathbf{CC}(\text{fin})$ fails. Then, in view of Proposition 3.5 of \cite{Her}, there exists a sequence $(A_n)_{n\in\omega}$ of pairwise disjoint non-void finite sets such that the set $Z=\bigcup_{n\in\omega} A_n$ is uncountable. Let us equip $Z$ with its discrete topology. Then the collection $\{\bigcup_{m\in n} A_m: n\in\omega\}$ is a countable base for $\mathbf{FB}(Z)$. Then, by Theorem 4.7, the bornological universe $(Z, \mathbf{FB}(Z))$ is metrizable. This contradicts $(ii)$. 
\end{proof}

For a set $X$ and a cardinal number $\kappa$, let us use the notation $[X]^{\leq \kappa}$ for the collection of all subsets $A$ of $X$ such that $A$ is of cardinality at most $\kappa$ and the notation $[X]^{<\kappa}$ for the collection of all subsets of $X$ that are of cardinality $<\kappa$. (cf. Definition I.13.19 of \cite{Ku2}). Then $[X]^{<\omega}=\mathbf{FB}(X)$, while $[X]^{\leq\omega}$ is the bornology of all at most countable subsets of $X$. 
 
The proof to the following interesting theorem is somewhat more complicated than to Theorem 7.1. 
\begin{thm} 
Equivalent are:
\begin{enumerate}
\item[(i)] for every sequence $(X_n)_{n\in\omega}$ of non-void at most countable sets $X_n$, the product $\prod_{n\in\omega}X_n$ is non-void;
\item[(ii)] for every discrete space $X$, the bornological universe $(X, [X]^{\leq \omega})$ is metrizable if and only if $X$ is countable.
\end{enumerate}
\end{thm}
\begin{proof} 
Assume $(i)$. Let $X$ be a discrete space such that the bornological universe $(X, [X]^{\leq \omega})$ is metrizable. Then, by Theorem 4.7, there exists a countable base $\mc{B}=\{ X_n: n\in\omega\}$ for the bornology $[X]^{\leq\omega}$. Suppose that $X$ is uncountable. We may assume that $X_n\subseteq X_{n+1}$ and that $X_n\neq X_{n+1}$ for each $n\in\omega$. By $(i)$, there exists $x\in\prod_{n\in\omega}(X_{n+1}\setminus X_n)$. Then, for such an $x$, if $A=\{ x(n): n\in\omega\}$, then $A\in [X]^{\leq\omega}$, while there does not exist $n\in\omega$ such that $A\subseteq X_n$. This is impossible because $\mc{B}$ is a base for $[X]^{\leq\omega}$.  Therefore, $(i)$ implies $(ii)$.

Now, let us suppose that $(i)$ is false. Consider any sequence $(X_n)_{n\in\omega}$ of non-empty countable sets such that $\prod_{n\in\omega}X_n=\emptyset$. For each $n\in\omega$, the set $Y_n=\prod_{i\in n+1}X_i$ is countable and non-empty. In much the same way as in the proof to Theorem 2.12 of \cite{Her}, we can show that there does not exist an infinite set $M\subseteq \omega$ such that $\prod_{n\in M} Y_n\neq\emptyset$. Let $Y=\bigcup_{n\in\omega}Y_n$ and let $f:\omega\to Y$ be an injection. Then the set $M_f=\{n\in\omega: f(\omega)\cap Y_n\neq\emptyset\}$ is finite. This proves that $Y$ is uncountable and if $A_n=\bigcup_{m\in n+1}Y_m$ for $n\in\omega$, then the collection $\mc{A}=\{A_n: n\in\omega\}$ is a countable base for $[Y]^{\leq\omega}$. If we equip $Y$ with its discrete topology, we will obtain that $(ii)$ is false. Hence $(ii)$ implies $(i)$.
\end{proof}

\begin{rem} It is unknown to us whether there is a model for $\mathbf{ZF}$ in which $\mathbf{CUT}$ fails (cf. \cite{Her}), while condition $(i)$ of Theorem 7.2 is satisfied. 
\end{rem}

\begin{rem} It is evident that conditions $(1)$ and $(2)$ of Theorem 2.6 of \cite{GM} are equivalent in \textbf{ZF}. In view of our Theorem 6.5 and the proof of $(3)\Rightarrow (1)$ of Theorem 2.6 given in \cite{GM}, we have that $(3)\Rightarrow (1)$ of Theorem 2.6 of \cite{GM} holds true in \textbf{ZF}. Since $\mathbf{CC}(\text{fin})$ is relatively independent of \textbf{ZF}, it follows from our Theorem 7.1 that Theorem 2.6 of \cite{GM} is relatively independent of $\mathbf{ZF}$. If $\mathbf{M}$ is a model for $\mathbf{ZF}+\neg \mathbf{CC}(\text{fin})$, then Theorems 7.1 and 6.5  show that there exists in $\mathbf M$ an uncountable metric space $(X, d)$ such that $\mathbf{FB}(X)$ is uniformly metrizable with respect to $d$, so Theorem 2.6 of \cite{GM} fails in $\mathbf{M}$. Now, we can deduce from Proposition 3.5 of \cite{Her} that Theorem 2.6 of \cite{GM} is equivalent with $\mathbf{CC}(\text{fin})$.
\end{rem}

\begin{rem} Let us notice that both $(1)\Leftrightarrow (2)$ and $(3)\Rightarrow (1)$ of Theorem 3.1 of \cite{GM} hold true in $\mathbf{ZF}$. Unfortunately, Theorem 3.1 of \cite{GM} is relatively independent of $\mathbf{ZF}$. Namely, in much the same way as in Remark 7.4, we can show that in every model $\mathbf{M}$ for $\mathbf{ZF}+\neg \mathbf{CC}(\text{fin})$, there exists an uncountable set $X$ such that, for the discrete metric $d$ in $X$, the bornology $\mathbf{CB}_d(X)$ is uniformly metrizable with respect to $d$, while $(X, d)$ is not Lindel\"of but it is obviously uniformly locally compact.
\end{rem}

\begin{rem} As Gutierres showed in \cite{Gut}, while working with completions of metric spaces, one must be more careful in \textbf{ZF} than in $\mathbf{ZF+CC}$. Let us observe that if $\mathbf{M}$ is a model for $\mathbf{ZF}$ such that there is in $\mathbf{M}$ an uncountable set $X$ such that $\mathbf{FB}(X)$ is uniformly metrizable with respect to the discrete metric $d$ in $X$, then $\mathbf{TB}_d(X)=\mathbf{FB}(X)=\mathbf{BB}_d(X)$ is uniformly metrizable, while $(X, d)$ is neither Lindel\"of nor Bourbaki-separable. Therefore, Theorems 4.2  and 5.8 of \cite{GM} fail in $\mathbf{M}$. In the light of our Theorem 7.1 and of the fact that $\mathbf{CC}(\text{fin})$ is relatively independent of $\mathbf{ZF}$, Theorems 4.2 and 5.8 of \cite{GM} are relatively independent of \textbf{ZF}.
\end{rem}

Since many articles about bornologies have been published so far,  it may take a lot of time to investigate which of the theorems in the articles can fail in some models for $\mathbf{ZF}$. There are theorems about connections between bornologies and realcompactifications that have already appeared in print (cf. \cite{Vr2}) and they seem to be unprovable in $\mathbf{ZF}$. In view of Theorem 10.12 of \cite{PW}, perhaps, some of them can be proved in $\mathbf{ZF+UFT}$ where $\mathbf{UFT}$ stands for the Ultrafilter Theorem  (cf. \cite{Her}). Let us leave it as an open problem which of the theorems about bornologies that have been proved by other authors in $\mathbf{ZFC}$ may fail in models for $\mathbf{ZF}$ and which of them can be proved under weaker assumptions than $\mathbf{ZFC}$. We have given only a partial solution to this problem. 

\section{Compact bornologies in quasi-metric spaces}

In the light of Remark 7.5, Theorem 3.1 of \cite{GM} may fail in a model for $\mathbf{ZF}$.  We are going to prove in $\mathbf{ZF}$ its modified version for compact bornologies in quasi-metric spaces. 

For a topological space $(X, \tau)$, let $\mathbf{CB}_{\tau}(X)$ be the bornology in $X$ generated by the collection of all compact subsets of $(X, \tau)$. If it is useful, we shall use the notation $\mathbf{CB}((X, \tau))$ for $\mathbf{CB}_{\tau}(X)$.

\begin{defi} Let $d$ be a quasi-metric in $X$. 
\begin{enumerate}
\item[(i)] We denote by $\mathbf{CB}_d(X)$ the bornology $\mathbf{CB}_{\tau(d)}(X)$.
\item[(ii)] We say that $X$ is \textbf{uniformly locally compact with respect to} $d$ if there exists $\delta\in (0; +\infty)$ such that $B_d(x, \delta)\in\mathbf{CB}_d(X)$ for each $x\in X$. 
\end{enumerate} 
\end{defi}  

The following example shows that, contrary to compact bornologies in metric spaces,  it may happen that, for a quasi-metric $d$ in $X$, there is a set $A\in\mathbf{CB}_d(X)$ such that $\text{cl}_{\tau(d)}A\notin\mathbf{CB}_d(X)$. 

\begin{exam}
Let us consider the set  $X=X_1\cup X_2$ where $X_1=\{ \frac{1}{2^{2n}}: n\in\omega\}$ and $X_2=\{\frac{1}{2^{2n+1}}: n\in\omega\}$. Let $x, y\in X$. If $x=y$, we put $d(x, y)=0$. When $x\neq y$, we put $d(x, y)=1$ if either $x, y\in X_1$ or $x,y\in X_2$, or $x\in X_1, y\in X_2$; moreover, we put $d(x, y)= y$ if $x\in X_2, y\in X_1$. In this way, we have defined a quasi-metric on $X$ such that, for each $y\in X_2$, the set $A_y=\{y\}\cup X_1$ is compact in $(X, \tau(d))$, while  $\text{cl}_{\tau(d)}A_y=X\notin\mathbf{CB}_d(X)$. 
\end{exam}

\begin{defi} We say that a topological space $(X, \tau)$ is $\sigma$-$\mathbf{CB}$ if there exists a countable collection $\mc{A}\subseteq\mathbf{CB}_{\tau}(X)$ such that $X=\bigcup\mc{A}$. 
\end{defi}

\begin{rem} Clearly, it holds true in $\mathbf{ZF}$ that every $\sigma$-compact space is $\sigma$-$\mathbf{CB}$ and every $\sigma$-$\mathbf{CB}$ Hausdorff space is $\sigma$-compact. In every model for $\mathbf{ZF+CC}$, a topological space is $\sigma$-compact if and only if it is $\sigma$-$\mathbf{CB}$. We do not know whether there is a model for $\mathbf{ZF}+\neg \mathbf{CC}$ in which  a  topological space can be simultaneously $\sigma$-$\mathbf{CB}$ and not $\sigma$-compact.
\end{rem} 

\begin{thm} Let $d$ be a (quasi)-metric in $X$. Then the following conditions are equivalent:
\begin{enumerate}
\item[(i)] $\mathbf{CB}_d(X)$ is uniformly (quasi)-metrizable with respect to $d$;
\item[(ii)] $X$ is uniformly locally compact with respect to $d$ and $(X, \tau(d))$ is $\sigma$-$\mathbf{CB}$. 
\end{enumerate}
\end{thm}

\begin{proof} Assume $(i)$. Let $\rho$ be a uniformly equivalent with $d$ quasi-metric in $X$ such that $\mathbf{CB}_d(X)$ is the collection of all $\rho$-bounded sets. There exists $\delta\in (0; +\infty)$ such that $\rho(x, y)<1$ whenever $d(x, y)<\delta$. Then, for each $x\in X$, we have  $B_d(x, \delta)\subseteq B_{\rho}(x, 1)\in \mathbf{CB}_d(X)$, so $X$ is uniformly locally compact with respect to $d$. Moreover, since, by Theorem 4.7, $\mathbf{CB}_d(X)$ has a countable base, we deduce that $(X, \tau(d))$ is $\sigma$-$\mathbf{CB}$.

Now, assume $(ii)$. Let $\delta\in (0; +\infty)$ be such that, for each $x\in X$, we have $B_d(x, \delta)\in \mathbf{CB}_d(X)$. Let $C$ be compact in $(X, \tau(d))$. It follows from the compactness of $C$ that there exists a finite set $K\subseteq C$ such that $C\subseteq \bigcup_{x\in K}B_d(x, \frac{\delta}{2})$. Then $[C]^{\frac{\delta}{2}}_d\subseteq\bigcup_{x\in K}[B_d(x,\frac{\delta}{2})]^{\frac{\delta}{2}}_d\subseteq\bigcup_{x\in K}B_d(x, \delta)\in \mathbf{CB}_d(X)$. Therefore, $[C]^{\frac{\delta}{2}}_d\in \mathbf{CB}_d(X)$ (cf. proof to 3.1 in \cite{GM}). This implies that $[C]^{\frac{\delta}{2}}_d\in \mathbf{CB}_d(X)$ whenever $C\in\mathbf{CB}_d(X)$. 

Let $\mc{A}=\{A_n: n\in\omega\}\subseteq\mathbf{CB}_d(X)$ be such that $X=\bigcup\mc{A}$. We may assume that $A_n\subseteq A_{n+1}$ for each $n\in\omega$. If $C$ is compact in $(X, \tau(d))$ then, since $C\subseteq \bigcup_{n\in\omega}[A_n]^{\frac{\delta}{2}}_d$, there exists $m\in\omega$ such that $C\subseteq \bigcup_{n\in m+1}[A_n]^{\frac{\delta}{2}}_d=[A_m]^{\frac{\delta}{2}}_d\in\mathbf{CB}_d(X)$. Therefore, the collection $\{[A_n]^{\frac{\delta}{2}}_d: n\in\omega\}$, is a countable base for $\mathbf{CB}_d(X)$.  This, together with the fact that $[C]^{\frac{\delta}{2}}_d\in \mathbf{CB}_d(X)$ whenever $C\in\mathbf{CB}_d(X)$, implies that there exists a subsequence $(B_n)_{n\in\omega}$ of the sequence $([A_n]^{\frac{\delta}{2}}_d)_{n\in\omega}$ such that $[B_n]^{\frac{\delta}{2}}_d\subseteq B_{n+1}$ for each $n\in\omega$. Then $\{ B_n: n\in\omega\}$ is a base for $\mathbf{CB}_d(X)$. This, together with Theorem 6.5, implies that  $(i)$ follows from $(ii)$.   
\end{proof}

\begin{exam}
Let $(X, d)$ be the quasi-metric space from Example 8.2. Then condition $(ii)$ of Theorem 8.5 is satisfied; hence,  the bornology $\mathbf{CB}_d(X)$ is uniformly quasi-metrizable with respect to $d$. We can also define a uniformly locally identical with $d$ quasi-metric $\rho$ in $X$ such that $\mathbf{CB}_d(X)=\mc{B}_{\rho}(X)$. To do this, let us consider $x, y\in X$. We put $\rho(x, y)=0$ if $x=y$. Now, suppose that $x\neq y$. Then $\rho(x, y)=1$ if $x, y\in X_1$. For $x\in X_2$ and $y\in X_1$, we define $\rho(x, y)=y$. Finally, for $x\in X$ and $y\in X_2$, we put $\rho(x, y)=\frac{1}{y}$. 
\end{exam}

Using similar arguments to the ones of the proof to Theorem 8.5, we deduce the following corollary:

\begin{cor}[cf. \cite{WJ}, Theorem 3.1 of \cite{GM} and  Corollary 3.3 of \cite{GM}]
 For every metric space $(X, d)$, it holds true in $\mathbf{ZF}$ that $\mathbf{CB}_d(X)$ is uniformly metrizable with respect to $d$ if and only if $(X, d)$ is both $\sigma$-compact and uniformly locally compact.
\end{cor}  

\section{Fundamental bornologies in gtses}

A new problem is to find an appropriate definition of (quasi)-metrizability for a generalized topological space (in abbreviation: a gts) in the sense of Delfs and Knebusch. Since the notion of a gts in this sense is rather complicated (cf. \cite{DK}, \cite{Pie1}) and it seems that it is still not commonly known to the mathematical community, let us recall it to make our paper more legible.  

\begin{defi}[cf. Definition 2.2.2 in \cite{Pie1}]
 A \textbf{generalized topological space} in the sense of Delfs and Knebusch (abbreviated to gts) is a triple $(X, \text{Op}_X, \text{Cov}_X)$ where $X$ is a set for which $\text{Op}_X\subseteq \mc{P}(X)$, while $\text{Cov}_X\subseteq \mc{P}(\text{Op}_X)$ and the following conditions are satisfied:
\begin{enumerate}
\item[(i)]  if $\mc{U}\subseteq \text{Op}_X$ and $\mc{U}$ is finite, then $\bigcup\mc{U}\in\text{Op}_X, \bigcap\mc{U}\in\text{Op}_X$ and  $\mc{U}\in \text{Cov}_X$;
\item[(ii)] if $\mc{U}\in\text{Cov}_X, V\in\text{Op}_X$ and $V\subseteq\bigcup\mc{U}$, then $\{ U\cap V: U\in\mc{U}\}\in\text{Cov}_X$;
\item[(iii)] if $\mc{U}\in\text{Cov}_X$ and, for each $U\in \mc{U}$, we have $\mc{V}(U)\in\text{Cov}_X$ such that $\bigcup\mc{V}(U)=U$, then $\bigcup_{U\in\mc{U}}\mc{V}(U)\in\text{Cov}_X$;
\item[(iv)] if $\mc{U}\subseteq \text{Op}_X$ and $\mc{V}\in \text{Cov}_X$ are such that $\bigcup\mc{V}=\bigcup\mc{U}$ and, for each $V\in\mc{V}$ there exists $U\in\mc{U}$ such that $V\subseteq U$, then $\mc{U}\in\text{Cov}_X$;
\item[(v)] if $\mc{U}\in\text{Cov}_X$, $V\subseteq\bigcup_{U\in\mc{U}}U$ and, for each $U\in\mc{U}$, we have $V\cap U\in\text{Op}_X$, then $V\in\text{Op}_X$.
\end{enumerate}
\end{defi}

\begin{rem} If  $(X, \text{Op}_X, \text{Cov}_X)$ is a gts, then $\text{Op}_X=\bigcup\text{Cov}_X$ and, therefore, we can identify the gts with the ordered pair $(X, \text{Cov}_X)$ (cf. \cite{Pie1}, \cite{PW}). If this is not misleading, we shall denote a gts $(X, \text{Cov}_X)$ by $X$.
\end{rem}
 
As far as gtses are concerned, we shall use the terminology of \cite{DK}, \cite{Pie1}-\cite{Pie2} and \cite{PW}.

\begin{defi}[cf. \cite{Pie1}] If $ X=(X, \text{Cov}_X)$ and $Y=(Y, \text{Cov}_Y)$ are gtses, then:
\begin{enumerate}
\item[(i)] a set $U\subseteq X$ is called \textbf{open} in the gts $X$ if $U\in\text{Op}_X$;
\item[(ii)] the collection $\text{Cov}_X$ is \textbf{the generalized topology} in $X$;
\item[(iii)] an \textbf{admissible open family} in the gts $X$ is a member of $\text{Cov}_X$;
\item[(iv)] a mapping $f:Y\to X$ is $(\text{Cov}_Y, \text{Cov}_X)$-\textbf{strictly continuous} (in abbreviation: strictly continuous) if, for each $\mc{U}\in \text{Cov}_X$, we have $\{ f^{-1}(U):U\in \mc{U}\}\in \text{Cov}_Y$.
\end{enumerate}
\end{defi}

In this section, let us have a brief look at very natural bornologies in generalized topological spaces. In the next section, we apply the bornologies in gtses to our concepts of (quasi)-metrizability in the category $\mathbf{GTS}$ of generalized topological spaces and strictly continuous mappings.  

\begin{defi}[cf. Definitions 2.2.13 and 2.2.25 of \cite{Pie1}]
 If $K$ is a subset of a set $X$, then we say that a family $\mc{U}\subseteq\mc{P}(X)$ is \textbf{essentially finite on} $K$ if there exists a finite $\mc{V}\subseteq\mc{U}$ such that $K\cap\bigcup\mc{U}\subseteq\bigcup\mc{V}$.
\end{defi}

\begin{defi}[cf. Definition 2.2.25 of \cite{Pie1}]
 If $X=(X,\text{Cov}_X)$ is a gts, then a set $K\subseteq X$ is called \textbf{small in the gts $X$} if each family $\mc{U}\in\text{Cov}_X$ is essentially finite on $K$.
\end{defi}

The collection of all small sets of a gts $X$ is a bornology in $X$ (cf. Fact 2.2.30 of \cite{Pie1}). 

\begin{defi} 
For a gts $X$, \textbf{the small bornology} of $X$ is the collection $\mathbf{Sm}(X)$ of all small sets in $X$.
\end{defi}

$\mathbf{Sm}(X)$ was denoted by $\text{Sm}_X$ in \cite{Pie1} but, since we are inspired by \cite{GM} and we use the notation of \cite{GM}, we have replaced $\text{Sm}_X$ by $\mathbf{Sm}(X)$ partly for elegance, partly for convenience. 

\begin{defi}[cf. Definition 3.2 of \cite{PW}]
 If $X$ is a gts, we call a set $A\subseteq X$ \textbf{admissibly compact} in $X$ if, for each $\mc{U}\in\text{Cov}_X$ such that $A\subseteq\bigcup\mc{U}$, there exists a finite $\mc{V}\subseteq \mc{U}$ such that $A\subseteq\bigcup\mc{V}$ .
\end{defi}

\begin{defi} For a gts $X$, \textbf{the admissibly compact bornology} of $X$ is the collection $\mathbf{ACB}(X)$ of all subsets of admissibly compact sets of the gts $X$. 
\end{defi}

For a collection $\mc{A}$ of subsets of a set $X$, we denote by $\tau(\mc{A})$ the weakest among all topologies in $X$ that contain $\mc{A}$. For a gts $(X, \text{Op}_X, \text{Cov}_X)$, we call the topological space $X_{top}=(X, \tau(\text{Op}_X))$ \textbf{the topologization of the gts} $X$ (cf. \cite{PW}).

\begin{defi} Let $X$ be a gts. We say that a set $A$ is \textbf{topologically compact} in $X$ if $A$ is compact in $X_{top}$ (cf.  Definition 3.2 of \cite{PW}).
The compact bornology $\mathbf{CB}(X_{top})$ will be called \textbf{the compact bornology of the gts} $X$ and it will be denoted by $\mathbf{CB}(X)$.
\end{defi}

\begin{f} For every gts $X$, the inclusion $(\mathbf{Sm}(X)\cup\mathbf{CB}(X))\subseteq \mathbf{ACB}(X)$ holds.
\end{f}

In general,  the collections $\mathbf{Sm}(X)\cup \mathbf{CB}(X)$ and $\mathbf{ACB}(X)$ can be distinct and neither $\mathbf{Sm}(X)\subseteq\mathbf{CB}(X)$ nor $\mathbf{CB}(X)\subseteq\mathbf{Sm}(X)$.

\begin{exam} \label{comp-born}
For $X=\mathbb{R}\times\{0, 1\}$, let $\text{Op}_X$ be the natural topology in $X$ inherited from the usual topology of $\mathbb{R}$ and let $\text{Cov}_X$ be the collection of all families $\mc{U}\subseteq \text{Op}_X$ such that $\mc{U}$ is essentially finite on $\mathbb{R}\times\{0\}$. Then, for  $A=[0; 1]\times\{1\}$ and $B=\mathbb{R}\times\{0\}$, we have $A\in\mathbf{CB}(X)\setminus\mathbf{Sm}(X)$  and $B\in\mathbf{Sm}(X)\setminus\mathbf{CB}(X)$, while $A\cup B\in\mathbf{ACB}(X)\setminus (\mathbf{CB}(X)\cup\mathbf{Sm}(X))$ . 
\end{exam}

For a set $X$ and a collection $\Psi\subseteq\mc{P}^2(X)$, we denote by $\langle \Psi\rangle_X$ the smallest among  generalized topologies in $X$ that contain $\Psi$. If $\mc{A}\subseteq \mc{P}(X)$, let $\text{EssCount}(\mc{A})$ be the collection of all essentially countable subfamilies of $\mc{A}$. We recall that $\text{EssFin}(\mc{A})$ is the collection of all essentially finite subfamilies of $\mc{A}$ (cf. \cite{Pie1}-\cite{Pie2} and \cite{PW}).

\begin{f}[cf. Examples 2.2.35 and 2.2.14(8) of \cite{Pie1}]
 Let $(X, \tau)$ be a topological space. That $\text{EssFin}(\tau)$ is a generalized topology in $X$ is true in $\mathbf{ZF}$. That $\text{EssCount}(\tau)$ is a generalized topology in $X$ is true in $\mathbf{ZF+CC}$.
\end{f}

\begin{rem}
It is unprovable in $\mathbf{ZF}$ that, for every topological space $(X, \tau)$,  the collection $\text{EssCount}(\tau)$ is a generalized topology in $X$. Namely, let $\mathbf{M}$ be a model for $\mathbf{ZF+\neg CC}(\text{fin})$. In view of the proof to Theorem 7.1, there exists in $\mathbf{M}$ an uncountable set $X$ such that $X$ is a countable union of finite sets. Let $\tau=\mc{P}(X)$. If $\text{EssCount}(\tau)$ were a generalized topology in $X$, the family of all singletons of $X$ would belong to $\text{EssCount}(\tau)$ which is impossible since $X$ is uncountable. 
\end{rem}

Let us observe that, for the gts $X$ from Example \ref{comp-born}, the admissibly compact bornology of $X$ is generated by $\mathbf{CB}(X)\cup\mathbf{Sm}(X)$.
That not every gts may share this property is shown by the following example:

\begin{exam} (\textbf{ZF+CC}) For $X=\omega_1$, let $\text{Op}_X$ be the topology induced by the usual linear order in $\omega_1$ and let $\text{Cov}_X=\text{EssCount}(\text{Op}_X)$. Then $\mathbf{Sm}(X)=\mathbf{FB}(X)\neq \mathbf{CB}(X)\neq \mathbf{ACB}(X)=\mc{P}(X)$. 
\end{exam} 

In what follows, for sets  $X, Y$ with $Y\subseteq X$ and for $\Psi\subseteq \mc{P}^2(X)$, we use the notation $\Psi\cap_2Y$  from \cite{Pie1} for the collection of all families $\mc{U}\cap_1 Y=\{U\cap Y: U\in\mc{U}\}$ where $\mc{U}\in\Psi$. We want to describe $\langle\Psi\cap_2 Y\rangle_Y$ more precisely in the case when $\Psi\cap_2 Y\subseteq \text{EssFin}(\mc{P}(Y))$. To do this, we need the concept of a complete ring of sets in $Y$ that was of frequent use in \cite{PW}. Namely, \textbf{ a complete ring in $Y$} is a collection $\mc{C}\subseteq \mc{P}(Y)$ such that $\emptyset, Y\in\mc{C}$, while $\mc{C}$ is closed under finite unions and under finite intersections. For $\mc{A}\subseteq \mc{P}(Y)$, let $L_Y[\mc{A}]$ be the intersection of all complete rings in $Y$ that contain $\mc{A}$. 

\begin{prop} For a set $X$, let $\Psi\subseteq \mc{P}^2(X)$. Suppose that $Y\subseteq X$ and that each family from $\Psi$ is essentially finite on $Y$. Then  the following conditions are satisfied: 
\begin{enumerate}
\item[(i)] $\langle \Psi\cap_2 Y\rangle_Y=\text{EssFin}( L_Y[\bigcup(\Psi\cap_2 Y)])=\\\text{EssFin}(\bigcup\langle\Psi\cap_2 Y\rangle_Y)=\text{EssFin}(\bigcup\langle \Psi\rangle_X)\cap_2 Y$;
\item[(ii)] each family from $\langle\Psi\rangle_X$ is essentially finite on $Y$.
\end{enumerate}
\end{prop}
\begin{proof} By applying Proposition 2.2.37 of \cite{Pie1} to the mapping $\text{id}_Y: Y\to X$, we obtain the inclusion $\langle \Psi\rangle_X\cap_2 Y\subseteq \langle\Psi\cap_2 Y\rangle_Y$ which, together with $(i)$, implies $(ii)$. To prove $(i)$, let us put $\mc{G}_0=\langle\Psi\cap_2 Y\rangle_Y, \mc{G}_1=\text{EssFin}( L_Y[\bigcup(\Psi\cap_2 Y)]), \mc{G}_2=\text{EssFin}(\bigcup\mc{G}_0)$ and $\mc{G}_3=\text{EssFin}(\bigcup\langle \Psi\rangle_X)\cap_2 Y$. Obviously, $\mc{G}_0, \mc{G}_1$ and $\mc{G}_2$ are generalized topologies in $Y$. By Proposition 2.2.53 of \cite{Pie1}, the collection $\mc{G}_3$ is also a generalized topology in $Y$. Since $\Psi\cap_2 Y\subseteq \mc{G}_1$ and $L_Y[\bigcup(\Psi\cap_2 Y)]\subseteq \bigcup\mc{G}_0$, we have $\mc{G}_0\subseteq \mc{G}_1\subseteq\mc{G}_2\subseteq\mc{G}_0$. It follows from the inclusion $\langle \Psi\rangle_X\cap_2 Y\subseteq \mc{G}_0$ that $\mc{G}_3\subseteq \mc{G}_0$. 
Since $\bigcup (\langle\Psi\rangle_X \cap_2 Y)$   is a complete ring of subsets of $Y$, we get  $\mc{G}_1\subseteq \mc{G}_3$. This completes our proof to $(i)$.  
\end{proof}

\begin{defi} If $X=(X, \text{Op}, \text{Cov})$ is a gts, then:
\begin{enumerate}
\item[(i)] \textbf{the partial topologization of} $X$  is the gts $X_{pt}=(X, (\text{Op})_{pt}, (\text{Cov})_{pt})$ where $(\text{Op})_{pt}=\tau(\text{Op})$ and $(\text{Cov})_{pt}=\langle \text{Cov}\cup\text{EssFin}(\tau(\text{Op}))\rangle_X$ (cf. Definition 4.1 of \cite{PW});
\item[(ii)] the gts $X$ is called \textbf{partially topological} if $X=X_{pt}$ (cf. Definition 2.2.4 of \cite{Pie1});
\item[(iii)] $\mathbf{GTS}_{pt}$ is the category of all partially topological spaces and strictly continuous mappings, while the mapping $pt:\mathbf{GTS}\to\mathbf{GTS}_{pt}$ is \textbf{the functor of partial topologization} defined by: $pt(X)=X_{pt}$ for every gts $X$ and $pt(f)=f$ for every morphism in $\mathbf{GTS}$ (cf. \cite{AHS}, \cite{Pie1} and Definition 4.2 of \cite{PW}).
\end{enumerate}
\end{defi}

\begin{prop}\label{bornologie-pt}
Let $X$ be a gts. Then $\mathbf{Sm}(X)=\mathbf{Sm}(X_{pt})$, $\mathbf{CB}(X)=\mathbf{CB}(X_{pt})$ and $\mathbf{ACB}(X_{pt})\subseteq\mathbf{ACB}(X)$.
\end{prop}
\begin{proof}
The equality $\mathbf{CB}(X)=\mathbf{CB}(X_{pt})$ and both the inclusions $\mathbf{Sm}(X_{pt})\subseteq\mathbf{Sm}(X)$ and $\mathbf{ACB}(X_{pt})\subseteq\mathbf{ACB}(X)$ are trivial. Let $X=(X, \text{Op}_X, \text{Cov}_X)$ and let $\Psi=\text{Cov}_X\cup\text{EssFin}(\tau(\text{Op}_X))$. Suppose that $Y\in\mathbf{Sm}(X)$. Since each family from $\Psi$ is essentially finite on $Y$, we infer from Proposition 9.15 that $Y\in\mathbf{Sm}(X_{pt})$.
\end{proof}

\begin{defi}
\textbf{A generalized bornological universe} is an ordered pair $((X,\text{Op}_X, \text{Cov}_X), \mc{B})$ such that $(X,\text{Op}_X, \text{Cov}_X)$ is a gts, while $\mc{B}$ is a bornology in $X$.
\end{defi}

\begin{defi}[cf. Proposition 2.2.71 of \cite{Pie1}] Let $\text{Op}_X$ be a complete ring of subsets of a set $X$. Then:
\begin{enumerate}
\item[(i)] for a collection $\mc{B}\subseteq \mc{P}(X)$, we define
$$\text{EF}(\text{Op}_X, \mc{B})=\{\mc{U}\subseteq\text{Op}_X: \forall 
_{A\in\mc{B}} \{A\cap U: U\in\mc{U}\}\in\text{EssFin}(\mc{P}(A))\};$$
\item[(ii)] for a topology $\tau$ in $X$ and for a bornology $\mc{B}$  in $X$, 
\textbf{the gts induced by the bornological universe} $((X, \tau), \mc{B})$ is the triple $\text{gts}((X, \tau),\mc{B})=(X, \tau,\text{EF}(\tau, \mc{B}))$.
\end{enumerate}
\end{defi}

In the light of the proof to Proposition 2.1.31 in \cite{Pie2}, we have the following fact:

\begin{f} Suppose that $((X, \tau), \mc{B})$ is a bornological universe such that $\tau\cap\mc{B}$ is a base for $\mc{B}$. Then $\mathbf{Sm}((X, \tau, \text{EF}(\tau, \mc{B})))=\mc{B}$.
\end{f}

\begin{defi}[cf. Example 2.1.12 of \cite{Pie2}]
 For a (quasi)-metric $d$ on a set $X$, the triple $(X, \tau(d), \text{EF}(\tau(d), \mc{B}_d(X)))$ will be called \textbf{ the gts induced by the (quasi)-metric} $d$.
\end{defi}
 
\begin{f}[cf. Example 2.1.12 of \cite{Pie2}]
 If $d$ is a quasi-metric on a set $X$,  then $ \text{EF}(\tau(d), \mc{B}_d(X))$ is a generalized topology in $X$ and  
 $$\mathbf{Sm}((X, \text{EF}(\tau(d), \mc{B}_d(X)) )=\mc{B}_d(X).$$
\end{f}

\section {$\mc{B}$-(quasi)-metrization of gtses}

\begin{defi} Suppose that $(X, \mc{B})$ is a generalized bornological universe. Then we say  that the gts $X$ is \textbf{$\mc{B}$-(quasi)-metrizable} or \textbf{(quasi)-metrizable with respect to $\mc{B}$} if the bornological universe $(X_{top}, \mc{B})$ is (quasi)-metrizable. 
\end{defi}

\begin{defi} 
Let $X$ be a gts and let $\mc{S}$ be either $\mathbf{CB}$ or $\mathbf{ACB}$, or $\mathbf{Sm}$. Then we say that $X$ is $\mc{S}$-(quasi)-metrizable if $X$ is (quasi)-metrizable with respect to $\mc{S}(X)$.
\end{defi}

With Proposition \ref{bornologie-pt} in hand, we can immediately deduce that the following proposition holds:

\begin{prop} Let $X$ be a gts and let $\mc{S}$ be either $\mathbf{CB}$ or $\mathbf{Sm}$. Then $X$ is $\mc{S}$-(quasi)-metrizable if and only if $X_{pt}$ is $\mc{S}$-(quasi)-metrizable. 
\end{prop}

\begin{rem} If $X$ is a gts, then the $\mathbf{ACB}$-(quasi)-metrizability of $X_{pt}$ is the (quasi)-metrizability of $X_{pt}$ with respect to $\mathbf{ACB}(X_{pt})$, while the $\mathbf{ACB}$-quasi-metrizability of $X$ is equivalent to the (quasi)-metrizability of $X_{pt}$ with respect to $\mathbf{ACB}(X)$. We do not know whether the $\mathbf{ACB}$-(quasi)-metrizability of $X$ is equivalent to the $\mathbf{ACB}$-(quasi)-metrizability of $X_{pt}$.
\end{rem} 

\begin{defi} 
A gts $X=(X, \text{Op}_X, \text{Cov}_X)$ is called:
\begin{enumerate}
\item[(i)] \textbf{locally small} if there exists $\mc{U}\in\text{Cov}_X$ such that $\mc{U}\subseteq \mathbf{Sm}(X)$ and $X=\bigcup\mc{U}$ (cf. Definition 2.1.1 of \cite{Pie2});
\item[(ii)] \textbf{weakly locally small} if there exists a collection $\mc{U}\subseteq \text{Op}_X\cap\mathbf{Sm}(X)$ such that $X=\bigcup\mc{U}$.
\end{enumerate}
\end{defi}

Our next theorem says about the form of the partial topologization of an $\mathbf{Sm}$-(quasi)-metrizable gts $X$ when $X_{pt}$ is locally small. 

\begin{thm} Suppose that $X=(X, \text{Op}, \text{Cov})$ is a gts such that its partial topologization $X_{pt}=(X,\text{Op}_{pt}, \text{Cov}_{pt})$ is locally small. Then the following conditions are equivalent:
\begin{enumerate}
\item[(i)]$X$ is $\mathbf{Sm}$-(quasi)-metrizable;
\item[(ii)] $X_{pt}$ is induced by some (quasi)-metric $d$.
\end{enumerate}
\end{thm}
\begin{proof} In view of Proposition \ref{bornologie-pt}, we have $\mathbf{Sm}(X)=\mathbf{Sm}(X_{pt})$. In consequence, it it is obvious that if $X_{pt}$ is induced by a (quasi)-metric $d$, then $X$ is $\mathbf{Sm}$-(quasi)-metrizable.  Assume that $X$ is $\mathbf{Sm}$-(quasi)-metrizable and that $d$ is a (quasi)-metric on $X$ such that $\tau(\text{Op})=\tau(d)$ and $\mathbf{Sm}(X_{pt})$ is the collection of all $d$-bounded sets. Since $X_{pt}$ is locally small, it follows from Proposition 2.1.18 of \cite{Pie2} that $X_{pt}$ is induced by $d$. 
\end{proof} 

\begin{f} If a gts $X$ is induced by a (quasi)-metric, then $X$ is locally small and partially topological.
\end{f}

\begin{f} \begin{enumerate} \item[(i)]If $X$ is a locally small gts, then $X_{pt}$ is locally small.
\item[(ii)] If a gts $X$ is such that $X_{pt}$ is locally small, then $X$ is weakly locally small.
\item[(iii)]  A gts $X$ is weakly locally small if and only if $X_{pt}$ is weakly locally small.
\end{enumerate}
\end{f}
We are going to present a pair of weakly locally small but not locally small gtses.
For $\Psi\subseteq \mc{P}^{2}(X)$,  we put $\Psi_0=\Psi$ and, for $n\in\omega$, assuming that the collection $\Psi_{n}\subseteq \mc{P}^2(X)$ has been defined, we put $\Psi_{n+1}=(\Psi_n)^+$ where $^+$ is the operator described in the proof of Proposition 2.2.37 in \cite{Pie1}. Then $\langle\Psi\rangle_X=\bigcup_{n\in\omega}\Psi_n$. The symbols $\cup_1, \cap_1, \cup_2, \cap_2$ have the same meaning as in \cite{Pie1}. 

\begin{exam}$[\mathbf{ZF+CC}]$. Suppose that  $Y$ is an uncountable set. For $n\in\omega$, we put $Y_n=Y\times\{n\}$. Let $X=\bigcup_{n\in\omega} Y_n$, $\text{Op}_X=\mathbf{FB}(X)\cup\{ X\}$ and $\text{Cov}_X=\text{EF}(\text{Op}_X, \{ Y_n: n\in\omega\})$. The gts $X=(X, \text{Op}_X, \text{Cov}_X)$ is weakly locally small and not small. If $X$ were locally small, then  $Y_0$ would be a subset of a small open set (Fact 2.1.21 in \cite{Pie2}), so $Y_0$ would be finite. Hence, $X$ is not locally small. We have $\{ Y_n: n\in \omega\}\in\text{EF}(\tau(\text{Op}_X), \{Y_n: n\in\omega\})$ and all the sets $Y_n$ are small and open in $(X, \text{EF}(\tau(\text{Op}_X), \{Y_n: n\in\omega\}))$, so the gts $(X, \text{EF}(\tau(\text{Op}_X), \{Y_n: n\in\omega\}))$ is locally small. We put $\Psi=\text{Cov}_X\cup \text{EssFin}(\tau(\text{Op}_X))$. Then $pt(\text{Cov}_X)=\langle\Psi\rangle_X$ is the generalized topology of  $X_{pt}$. By Proposition 9.15, $\langle\Psi\rangle_X\subseteq  \text{EF}(\tau(\text{Op}_X), \{Y_n: n\in\omega\}) $.  Surprisingly, if $\mathbf{CC}$ holds, then  $X_{pt}$ is not locally small and, in consequence, $\langle \Psi\rangle_X\subset \text{EF}(\tau(\text{Op}_X), \{Y_n: n\in\omega\})$. To prove this, let us assume $\mathbf{ZF+CC}$. It is easy to observe the following facts:\\
\textbf{Fact 1.} $X\notin [X]^{\leq\omega}\cup_1 \mathbf{Sm}(X)$.\\
\textbf{Fact 2.}   Each $\Psi_n$ is closed with respect to restriction: $\Psi_n\cap_2 A\subseteq \Psi_n$ for $A\subseteq X$. \\
For $\mc{W}\subseteq\mc{P}(X)$, let us consider the following property: 

$\mathbf{P}(\mc{W})$: $\mc{W}$ has an uncountable member and $\mc{W}\subseteq [X]^{\leq\omega}\cup_1 \mathbf{Sm}(X)$.\\
For $n\in\omega$, let $T(n)$ be the statement:

$T(n)$:  if $\mc{W}\in\Psi_n$  has $\mathbf{P}(\mc{W})$, then $\mc{W}$ is essentially finite on $X\setminus A$ for some countable  $A\subseteq X$.\\
We are going to prove by induction that the following fact holds:\\
\textbf{Fact 3.} $T(n)$ is true for each $n\in \omega$.
\begin{proof}
Let $\mc{W}\in\Psi_0$ have property $\mathbf{P}(\mc{W})$.  Then, by Fact 1,  $X\notin\mc{W}$. Thus $\mc{W}\in\text{EssFin}(\tau(\text{Op}_X))$. Hence $T(0)$ holds. Suppose that $T(n)$ is true.  The \textit{finiteness}, \textit{stability}, and \textit{regularity} induction steps from the proof of Proposition 2.2.37 in \cite{Pie1} are obvious.

\textit{Transitivity step.}  Let $\mc{W}\in\Psi_{n+1}$ have property $\mathbf{P}(\mc{W})$. Suppose that $\mc{U}\in{\Psi}_n$ and $\{ \mc{V}(U): U\in\mc{U}\}\subseteq\Psi_{n}$ are such that $\mc{W}=\bigcup_{U\in\mc{U}}\mc{V}(U)$ and, for each $U\in\mc{U}$, we have $U=\bigcup\mc{V}(U)$. Consider any $U\in\mc{U}$.  If every member of $\mc{V}(U)$ is countable, then $U\in[X]^{\leq\omega}$ because $\mathbf{CC}$ holds and $\mc{V}(U)$ is essentially countable. 
Suppose  $\mc{V}(U)$ has an uncountable member. Since $\mc{V}(U)$ has property $\mathbf{P}(\mc{V}(U))$, it follows from the inductive assumption that there is a countable set $A(U)\subseteq X$ such that $\mc{V}(U)$ is essentially finite on $X\setminus A(U)$. Then $U\in [X]^{\leq\omega}\cup_1 \mathbf{Sm}(X)$ and $U$ is uncountable. 
The above implies that $\mc{U}$ has property $\mathbf{P}(\mc{U})$. By the assumption, there is a countable $A\subseteq X$ such that $\mc{U}$ is essentially finite on $X\setminus A$. Let $\mc{U}^{\ast}\subseteq \mc{U}$ be a finite family such that $\bigcup\mc{U}^{\ast}\setminus A=\bigcup\mc{U}\setminus A$. For each $U\in \mc{U}^{\ast}$, the set $U$ is countable or $\mc{V}(U)$ is essentially finite on $U\setminus A(U)$. This implies that there is a countable $A(\mc{W})$ such that $\mc{W}$ is essentially finite on $X\setminus A(\mc{W})$. 

\textit{Saturation step.}  Suppose that there exists $\mc{V}\in\Psi_n$ such that $\bigcup\mc{V}=\bigcup\mc{W}$ and, for each $V\in\mc{V}$, there is $W(V)\in\mc{W}$ such that $V\subseteq W(V)$. 
Since $\mc{W}\subseteq [X]^{\leq\omega}\cup_1\mathbf{Sm}(X)$, we have $\mc{V}\subseteq [X]^{\leq\omega}\cup_1 \mathbf{Sm}(X)$. Since $\mc{W}$ has an uncountable member and $\mc{V}$ is essentially countable, also $\mc{V}$ has an uncountable member and has property $\mathbf{P}(\mc{V})$. By the inductive assumption, there exists a countable $A(\mc{V})$ such that $\mc{V}$ is essentially finite on $X\setminus A(\mc{V})$. Then $\mc{W}$ is essentially finite on $X\setminus A(\mc{V})$, too. 
\end{proof}

Suppose that  $X_{pt}$ is locally small. There exists $\mc{W}\in pt(\text{Cov}_X)$ such that $\mc{W}\subseteq\mathbf{Sm}(X)$ and $X=\bigcup\mc{W}$. Since $X$ is uncountable and $\mc{W}$ is essentially countable, at least one member of $\mc{W}$ is uncountable, so $\mathbf{P}(\mc{W})$ holds true. By Fact 3, there exists a countable $A(\mc{W})$ such that $\mc{W}$ is essentially finite on $X\setminus A(\mc{W})$. Then $X\setminus A(\mc{W})\in\mathbf{Sm}(X)$. This is impossible by Fact 1.
\end{exam}
 The example above is not a solution to the following open problem:
 \begin{p} Is it true in $\mathbf{ZF}$ that if the partial topologization of a gts $X$ is locally small, then so is $X$?
 \end{p}

\begin{prop} Suppose that  $X=(X, \text{Op}_X, \text{Cov}_X)$ is a gts and $\mc{B}$ is a bornology in $X$. Then the following conditions are equivalent:
\begin{enumerate}
\item[(i)] the gts $X$ is (quasi)-metrizable with respect to $\mc{B}$;
\item[(ii)] the gts $(X, \text{EF}(\tau(\text{Op}_X), \mc{B}))$ is $\mathbf{Sm}$-(quasi)-metrizable and $\tau(\text{Op}_X)\cap\mc{B}$ is a base for $\mc{B}$.
\end{enumerate}
\end{prop}
\begin{proof} Assume that $(i)$ holds. Then, by Theorem 4.7, the collection $\tau(\text{Op}_X)\cap\mc{B}$ is a base for $\mc{B}$. It follows from Fact 9.20 that $\mc{B}=\mathbf{Sm}((X, \text{EF}(\tau(\text{Op}_X), \mc{B})))$. In consequence, $(i)$ implies $(ii)$. On the other hand,  we can use Fact 9.20 with  both Definitions 9.19 and 10.1 to infer that $(i)$ follows from $(ii)$.
\end{proof}

\begin{defi} Suppose that $(X, \mc{B})$ is a generalized bornological universe where $X=(X, \text{Op}_X, \text{Cov}_X)$.  Let us say that $X$ is \textbf{strongly $\mc{B}$-(quasi)-metrizable} if there exists a (quasi)-metric $d$ on $X$ such that $\mc{B}$ is the collection of all $d$-bounded sets and $\text{Op}_X=L_X[\{B_d(x, r): x\in X \wedge  r\in (0; +\infty)\}]$.
\end{defi}

\begin{defi} \textbf{A (quasi)-metric gts} is an ordered pair $(X, d)$ where $X=(X,\text{Op}_X, \text{Cov}_X)$ is a gts and $d$ is a (quasi)-metric in $X$ such that $\tau(d)=\tau(\text{Op}_X)$. 
\end{defi}

\begin{defi} Suppose that $(X, d)$ is a (quasi)-metric gts and that $\mc{B}$ is a bornology in $X$.  We say that $(X, d)$ is \textbf{uniformly $\mc{B}$-(quasi)-metrizable} or \textbf{uniformly (quasi)-metrizable with respect to $\mc{B}$} if the bornology $\mc{B}$ is uniformly (quasi)-metrizable with respect to $d$.
\end{defi}

\begin{rem}
For a bornology $\mc{B}$ in a gts $X$, one can find results in the previous sections that deliver necessary and sufficient conditions for $X$ to be (quasi)-metrizable with respect to $\mc{B}$  (see Theorems 4.7 and 4.15, as well as Corollaries 4.10 and 4.16) and for a metric gts $(X, d)$ to be uniformly (quasi)-metrizable with respect to $\mc{B}$ (see Theorems 6.5 and 8.5). 
\end{rem}
 
Let us use the real lines described in Definition 1.2 of \cite{PW} as our illuminating examples for the notions of (uniform) $\mc{B}$-(quasi)-metrizability in the category $\mathbf{GTS}$.  

\begin{exam}
Let $\tau_{nat}$ be the natural topology in $\mathbb{R}$. For $x,y\in \mathbb{R}$, we put $d_n(x,y)=\mid x-y\mid, d_{n, 1}(x,y)=\min\{ d_n(x, y), 1\}$ and 
$$d_n^+(x,y)=d_n(\Phi(x),\Phi(y)) \text{ where } \Phi(x)= \left\{
\begin{array}{ll} e^x, & x<0,\\
1+x, & x\ge 0.\end{array} \right. $$
Moreover, we define  $d^+_{n,1}(x, y)=\min\{d^+_n(x, y), 1\}$. Let us observe that the metrics $d_n$ and $d^+_n$ are equivalent but not uniformly equivalent.
\begin{enumerate} 
\item[(i)] We have $\mc{B}_{d_n}(\mb{R})=\mathbf{CB}_{\tau_{nat}}(\mb{R})$ and $\mc{B}_{d^+_n}(\mb{R})=\mathbf{UB}(\mb{R})$. Let us observe that, for a fixed $\delta\in (0; +\infty)$, there exists $n(\delta)\in\omega$ such that if $C_m=[-m; m]$ for $m\in\omega$ with $m>n(\delta)$,  then $(-\infty; m)\subseteq [C_m]^{\delta}_{d^+_n}$. This, together with Theorem 6.5, implies that $\mc{B}_{d_n}(\mb{R})$ is not uniformly quasi-metrizable with respect to $d^+_n$.  

\item[(ii)] For the usual topological real line $\mb{R}_{ut}$ (cf. Definition 1.2(i) of \cite{PW}), we have $\mathbf{FB}=\mathbf{Sm}\subset \mathbf{CB}=\mathbf{ACB}$ and $\text{int}_{\tau_{nat}}A=\emptyset$ for each $A\in\mathbf{Sm}(\mb{R}_{ut})$,  so the gts $\mb{R}_{ut}$ is not $\mathbf{Sm}$-quasi-metrizable and it is $\mathbf{ACB}$-metrizable by $d_n$. The metric gtses $(\mb{R}_{ut}, d_n)$ and $(\mb{R}_{ut}, d_{n,1})$ are $\mathbf{ACB}$-uniformly metrizable. It follows from (i) that the metric gtses $(\mb{R}_{ut}, d^+_n)$ is $(\mb{R}_{ut}, d^+_{n, 1})$ are not uniformly $\mathbf{ACB}$-quasi-metrizable. 

\item[(iii)] For the real lines $\mb{R}_{lst}$ and $\mb{R}_{lom}$ (cf. Definition 1.2(iv)-(v) of \cite{PW}), we have $pt(\mb{R}_{lom})=\mb{R}_{lst}$ and $\mathbf{Sm}=\mathbf{CB}=\mathbf{ACB}=\mc{B}_{d_n}(\mb{R})$. The metric gtses $(\mb{R}_{lst}, d_n)$ and $(\mb{R}_{lom}, d_n)$ are both uniformly $\mathbf{Sm}$-metrizable; however, none of the metric gtses $(\mb{R}_{lom}, d^+_n)$ and $(\mb{R}_{lst}, d^+_n)$ is uniformly $\mathbf{Sm}$-metrizable (see (i)).

\item[(iv)] For the real lines $\mb{R}_{l^+om}$ and $\mb{R}_{l^{+}st}$ (cf. Definition 1.2(vii)-(viii) of \cite{PW}),  we have $pt(\mb{R}_{l^+om})=\mb{R}_{l^+st}$ and $\mathbf{CB}=\mathbf{CB}_{\tau_{nat}}(\mb{R})\subset \mathbf{Sm}=\mathbf{ACB}=\mc{B}_{d^+_n}(\mb{R})$. Now, it is obvious that both the metric gtses $(\mb{R}_{l^+om}, d^+_n)$ and $(\mb{R}_{l^+st}, d^+_n)$ are uniformly $\mathbf{ACB}$-metrizable by the metric $d^+_n$. The gtses $\mb{R}_{l^+om}$ and $\mb{R}_{l^{+}st}$ are $\mathbf{Sm}$-metrizable. The metric gtses $(\mb{R}_{l^+om}, d_n)$ and $(\mb{R}_{l^+st}, d_n)$ are uniformly $\mathbf{Sm}$-metrizable and uniformly $\mathbf{ACB}$-metrizable by 
$d_u(x,y)=d_{n,1}(x,y)+|\max(y,0)-\max(x,0)|$.

\item[(v)] Let us consider the gtses $\mb{R}_{om}, \mb{R}_{slom}, \mb{R}_{rom}$ and $\mb{R}_{st}$ (cf. Definition 1.2(ii),(iii), (vi) and (x) of \cite{PW}). We have $pt(\mb{R}_{om})= pt(\mb{R}_{slom})= pt(\mb{R}_{rom})=\mb{R}_{st}$ and $\mathbf{CB}\subset \mathbf{Sm}=\mathbf{ACB}=\mc{P}(\mb{R})$. The real lines $\mb{R}_{om}, \mb{R}_{slom}, \mb{R}_{rom}$ and $\mb{R}_{st}$  are $\mathbf{Sm}$-metrizable by the metric $d_{n, 1}$ and they are $\mathbf{CB}$-metrizable by the metric $d_n$.

\item[(vi)] The gts $\mb{R}_{om}$  (cf. Definition 1.2(ii) of \cite{PW}) is strongly $\mathbf{Sm}$-metrizable by $d_{n, 1}$.  
\end{enumerate}
\end{exam}

 In connection with strong $\mathbf{Sm}$-(quasi)-metrizability, let us pose the following open problem:
 \begin{p}
 Find useful simultaneously necessary and sufficient conditions for a gts to be strongly $\mathbf{Sm}$-(quasi)-metrizable. 
 \end{p}
 
It might be helpful to have a look at several simple examples of gtses of type $(X, \text{EF}(\tau, \mc{B}))$ and compare them with Proposition 10.11. 
 
 \begin{exam} (\textbf{Gtses from the Sorgenfrey line}.)  Let us use  the topologies $\tau_{S, r}$ and $\tau_{S, l}$ considered in Example 4.12, as well as the quasi-metrics $\rho_{S}$,  $\rho_{S, 1}$  and $\rho_L$ defined in Example 4.12.
 \begin{enumerate}
 \item[(i)] The gts $(\mb{R}, \text{EF}(\tau_{S, r}, \mathbf{CB}_{\tau_{nat}}(\mb{R})))$ is $\mathbf{Sm}$-quasi-metrizable by the quasi-metric $\rho_{0}$ defined as follows:
  $$ \rho_{0}(x,y)=\left\{ \begin{array}{ll}
 y-x, & x\le y\\
 1+x-y, & x>y.\end{array}\right.$$
\item[(ii)] The gts $(\mb{R}, \text{EF}(\tau_{S, r}, \mathbf{UB}(\mb{R})))$ is $\mathbf{Sm}$-quasi-metrizable by $\rho_S$, while the gts $(\mb{R}, \text{EF}(\tau_{S, r}, \mathbf{LB}(\mb{R})))$ is $\mathbf{Sm}$-quasi-metrizable by $\rho_L$.
 \item[(iii)] The gts $(\mb{R}, \text{EF}(\tau_{S, r}, \mc{P}(\mb{R})))$ is $\mathbf{Sm}$-quasi-metrizable by $\rho_{S, 1}$.
 \item[(iv)] It follows from Theorem 4.7 that the gtses  $(\mb{R}, \text{EF}(\tau_{S, r}, \mathbf{FB}(\mb{R})))$  and $(\mb{R}, \text{EF}(\tau_{nat}, \mathbf{FB}(\mb{R})))$ are not $\mathbf{Sm}$-quasi-metrizable  because $\tau_{S, r}\cap\mathbf{FB}(\mb{R})$ is not a base for $\mathbf{FB}(\mb{R})$. 
 \end{enumerate}
 \end{exam}
 
\begin{exam}(\textbf{Quasi-metric gtses from the Sorgenfrey line}.) We use the same notation as in Example 10.18.
\begin{enumerate}
\item[(i)] The quasi-metric gts $((\mb{R}, \text{EF}(\tau_{S, r}, \mathbf{CB}_{\tau_{nat}}(\mb{R}))), \rho_{S})$ is uniformly $\mathbf{Sm}$-quasi-metrizable by $\rho_{0}$.
\item[(ii)] The quasi-metric gts $((\mb{R}, \text{EF}(\tau_{S, r}, \mathbf{UB}(\mb{R}))), \rho_{0})$ is uniformly $\mathbf{Sm}$-quasi-metrizable by $\rho_{S}$, while the quasi-metric gts $((\mb{R}, \text{EF}(\tau_{S, r}, \mathbf{LB}(\mb{R}))), \rho_{0})$ is uniformly $\mathbf{Sm}$-quasi-metrizable by $\rho_{L}$,
\item[(iii)] The quasi-metric gts $((\mb{R}, \text{EF}(\tau_{S, r}, \mc{P}(\mb{R}))), \rho_0)$ is uniformly $\mathbf{Sm}$-quasi-metrizable by $\min\{ \rho_0, 1\}$.
\end{enumerate}
\end{exam}

\begin{exam} Let us put $J=[0; 1]\times\{0\}$ and $J_{q}=\{q\}\times [0; 1]$. For $S=[0; 1]\cap\mathbb{Q}$, let $X=J\cup\bigcup_{q\in S}J_q$. We consider the collection $\mc{B}$ of all sets $A\subseteq X$ that have the property: there exists a finite $S(A)\subseteq S$ such that $A\subseteq J\cup\bigcup_{q\in S(A)}J_q$.
\begin{enumerate}
\item[(i)] Let $d_e$ be the Euclidean metric in $X$. Then, for each $A\in\mc{B}$, we have $\text{int}_{\tau(d_e)}A=\emptyset$, so, for every topology $\tau_2$ in $X$,  the bornology $\mc{B}$ is not $(\tau(d_e), \tau_2)$-proper. In consequence, the gts $(X,\text{EF}(\tau(d_e), \mc{B}))$ is not $\mathbf{Sm}$-quasi-metrizable.
\item[(ii)] We define another metric $\rho$ in $X$ as follows. For $x, y\in [0; 1]$ and $q, q'\in S$ with $q\neq q'$, we put $\rho((x,0), (y,0))=\mid x-y\mid, \rho((q, x),(q, y))=\mid x-y\mid $ and $\rho((q, x), (q', y))=x+\mid q-q'\mid +y$. Then, for each $q\in S$ and for any $a, b\in [0; 1]$ with $a<b$, we have $\{ q\}\times (a; b)=\text{int}_{\tau(\rho)}[\{q\}\times (a; b)]\in\mc{B}$. Since there does not exist $A\in\mc{B}$ such that $J\subseteq\text{int}_{\tau(\rho)}A$, we deduce that the gts  $(X,\text{EF}(\tau(\rho), \mc{B}))$ is not $\mathbf{Sm}$-quasi-metrizable. The space $(X, \tau(\rho))$ can be called \textbf{the comb with its hand $J$ and teeth $J_q$}, $q\in \mb{Q}$ (compare with Example IV.4.7 of \cite{Kn}).
\end{enumerate}
\end{exam}
\begin{rem}
One can easily reformulate Theorems 4.7 and 4.15 to get simultaneously necessary and sufficient conditions for a bornological biuniverse to be quasi-pseudometrizable. One can also use quasi-pseudometrics instead of quasi-metrics in Theorem 6.5 to obtain conditions equivalent with the uniform quasi-pseudometrizability  of a bornology with respect to a given quasi-pseudometric.
\end{rem}
\begin{exam}
The topological space $(\mb{R}, u)$ is not quasi-metrizable (since it is not $T_1$) but it is quasi-pseudometrizable by $\rho_u$ (see Section 6).  
\begin{enumerate}
\item[(i)] The gts $(\mb{R}, \text{EF}(u,\mathbf{UB}(\mb{R})))$ is $\mathbf{Sm}$-quasi-pseudometrizable by $\rho_u$.
\item[(ii)] For the gts $\mb{R}_{ul}=(\mb{R}, \text{EF}(u,\mathbf{LB}(\mb{R})))$ we have $\mathbf{Sm}(\mb{R}_{ul})=\mc{P}(\mb{R})$.
This is why $\mb{R}_{ul}$ is $\mathbf{Sm}$-quasi-pseudometrizable by $\rho_{u,1}=\min\{1,\rho_u\}$.
\item[(iii)] For the gts $\mb{R}_{ub}=(\mb{R}, \text{EF}(u,\mathbf{UB}(\mb{R})\cap\mathbf{LB}(\mb{R})))$ we have $\mathbf{Sm}(\mb{R}_{ub})=\mathbf{UB}(\mb{R})$. This is why $\mb{R}_{ub}$
 is $\mathbf{Sm}$-quasi-pseudometrizable by $\rho_u$. 
\item[(iv)] The gts $\mb{R}_{uf}=(\mb{R}, \text{EF}(u, \mathbf{FB}(\mb{R})))$ is not $\mathbf{LB}(\mb{R})$-quasi-pseudometrizable because $\text{int}_uA=\emptyset$ for each $A\in\mathbf{LB}(\mb{R})$. Here $\mathbf{Sm}(\mb{R}_{uf})$ is the collection of all sets $A\in\mathbf{UB}(\mb{R})$ such that every non-empty subset of $A$ has a maximal element. Similarly, $\mb{R}_{uf}$ is not $\mathbf{Sm}$-quasi-pseudometrizable.  Since $\mathbf{ACB}(\mb{R}_{uf})=\mathbf{CB}(\mb{R}_{uf})=\mathbf{UB}(\mb{R})$, the gts $\mb{R}_{uf}$ is $\mathbf{ACB}$-quasi-pseudometrizable by $\rho_u$. 
\end{enumerate}
\end{exam} 

\section{New topological categories}

The table of categories in \cite{AHS}, among other categories, says about the category $\mathbf{Top}$ of topological spaces, the category $\mathbf{BiTop}$ of bitopological spaces and about the category $\mathbf{Bor}$ of bornological sets. The categories $\mathbf{GTS}$, $\mathbf{GTS}_{pt}$, $\mathbf{SS}$ of small generalized topological spaces and $\mathbf{LSS}$ of locally small generalized topological spaces, as well as  $\mathbf{SS}_{pt}$ and $\mathbf{LSS}_{pt}$, were introduced in \cite{Pie1} and \cite{Pie2}. We pointed out in  \cite{PW} that, while working with categories and proper classes, a modification of $\mathbf{ZF}$ is required. We assume a suitably modified version of $\mathbf{ZF}$ suggested in \cite{PW}. 

In the light of  Proposition 9.17  and Fact 10.8, we can state the following:

\begin{f} The functor $pt$ of partial topologization preserves smallness and local smallness. More precisely:
\begin{enumerate}
\item[(i)] $pt$ restricted to $\mathbf{SS}$ maps $\mathbf{SS}$ onto $\mathbf{SS}_{pt}$;
\item[(ii)] $pt$ restricted to $\mathbf{LSS}$ maps $\mathbf{LSS}$ onto $\mathbf{LSS}_{pt}$. 
\end{enumerate}
\end{f}

All the categories $\mathbf{Top}, \mathbf{BiTop},  \mathbf{GTS}, \mathbf{GTS}_{pt}, \mathbf{SS}, \mathbf{SS}_{pt}$ and $\mathbf{Bor}$ are topological constructs (cf. \cite{AHS}, \cite{Sal}, \cite{Pie1},\cite{Pie2},  \cite{PW} and \cite{H-N}). Since $\mathbf{Top}$ and $\mathbf{Bor}$ are topological constructs, it is obvious that the category $\mathbf{Ubor}$ of bornological universes (cf. Remark 2.2.70 of \cite{Pie1}) is a topological construct, too.  Let us define several more categories and answer the question whether they are topological constructs.

\begin{defi}[cf. 1.2.1 in \cite{H-N}] Let $\mc{B}_X$ be a boundedness  in a set $X$ and let $\mc{B}_Y$ be a boundedness  in a set $Y$. We say that a mapping $f:X\to Y$ is $(\mc{B}_X, \mc{B}_Y)$\textbf{-bounded} (in abbreviation: \textbf{bounded}) if, for each $A\in\mc{B}_X$, we have $f(A)\in\mc{B}_Y$. 
\end{defi}

\begin{defi} Suppose that $((X,\tau_1^X, \tau_2^X), \mc{B}_X)$ and $((Y,\tau_1^Y, \tau_2^Y), \mc{B}_Y)$  are bornological biuniverses. We say that a mapping $f:X\to Y$ is a \textbf{bounded bicontinuous mapping} from $((X,\tau_1^X, \tau_2^X), \mc{B}_X)$ to $((Y,\tau_1^Y, \tau_2^Y), \mc{B}_Y)$ if $f$ is bicontinuous  with respect to $(\tau_1^X, \tau_2^X, \tau_1^Y, \tau_2^Y)$  and $f$ is $(\mc{B}_X, \mc{B}_Y)$-bounded. 
\end{defi}

\begin{defi} Suppose that $((X, \text{Cov}_X), \mc{B}_X)$ and $((Y, \text{Cov}_Y), \mc{B}_Y)$ are generalized bornological universes. We say that a mapping $f: X\to Y$ is a \textbf{bounded strictly continuous mapping} from $((X, \text{Cov}_X), \mc{B}_X)$ to $((Y, \text{Cov}_Y), \mc{B}_Y)$ if $f$ is both $(\mc{B}_X, \mc{B}_Y)$-bounded and $(\text{Cov}_X, \text{Cov}_Y)$-strictly continuous. 
\end{defi}

\begin{defi} A generalized bornological universe $((X, \text{Cov}_X), \mc{B})$ is called: 
\begin{enumerate}
\item[(i)] \textbf{partially topological} if the gts $(X, \text{Cov}_X)$ is partially topological;
\item[(ii)] \textbf{small} if the gts $(X, \text{Cov}_X)$ is small.
\end{enumerate}
\end{defi}

\begin{defi} We define the following categories:
\begin{enumerate}
\item[(i)] $\mathbf{BiUBor}$ where objects are bornological biuniverses and morphisms are bounded bicontinuous mappings;
\item[(ii)] $\mathbf{GeUBor}$ where objects are generalized bornological universes and morphisms are bounded strictly continuous mappings;
\item[(iii)] $\mathbf{Ge_{pt}UBor}$ where objects are partially topological generalized bornological universes and morphisms are bounded strictly continuous mappings;
\item[(iv)] $\mathbf{SmUBor}$ where objects are small generalized  bornological universes and morphisms are bounded strictly continuous mappings;
\item[(v)] $\mathbf{Sm_{pt}UBor}$ where objects are partially topological small generalized bornological universes and morphisms are bounded strictly continuous mappings.
\end{enumerate}
\end{defi} 

 \begin{prop}
 The categories defined in 11.6 are all topological constructs.
 \end{prop}
 \begin{proof} To check that, for instance,  $\mathbf{Ge_{pt}UBor}$ is a topological construct, we mimic the proof to Theorem 4.4 of \cite{PW}. Namely, let us consider a source $F=\{ f_i: i\in I\}$ of mappings $f_i: X\to Y_i$ indexed by a class $I$ where every $Y_i$ is a partially topological generalized bornological universe and $Y_i=((X_i, \text{Cov}_i), \mc{B}_i)$. Let $\text{Cov}_X$ be the $\mathbf{GTS}$-initial generalized topology for $F$ in $X$ (cf. Definition 4.3 of \cite{PW}) and let 
 $\mc{B}_X=\bigcap_{i\in I}\{ A\subseteq X: f_i(A)\in\mc{B}_i\}$.  
 For $X=((X, \text{Cov}_X), \mc{B}_X)$, let $X_{pt}=(pt((X, \text{Cov}_X)), \mc{B}_X)$. The canonical morphism  $id: X_{pt}\to X$ is such that all mappings $f_i\circ id$ are morphisms in $\mathbf{Ge_{pt}UBor}$. For any object $Z$ of $\mathbf{Ge_{pt}UBor}$ and a mapping $h:Z\to X_{pt}$, we can observe that if all $f_i\circ id\circ h$ with $i\in I$ are morphisms, then $id\circ h$ is a morphism of $\mathbf{GTS}$, so $pt(h)=h$ is a morphism of $\mathbf{GTS}_{pt}$. If all $f_i\circ id\circ h$ are bounded, then $pt(h)=h$ is bounded, too. That $\mathbf{BiUBor}, \mathbf{GeUBor}, \mathbf{SmUBor}$ and $\mathbf{Sm_{pt}UBor}$ are topological can be proved by using more or less similar arguments. 
 \end{proof} 
 
Some other topological constructs relevant to bornologies or to quasi-pseudometrics were considered in \cite{CL} and \cite{Vr1}.


\begin{thebibliography}{BOP} 
\normalsize 
\baselineskip=17pt 

\bibitem[Al]{Al} J. W. Alexander, \emph{On the concept of a topological space}, Proc. Nat. Acad. Sci. U. S. A. 25 (1939), 52--54.

\bibitem[An]{An} A. Andrikopoulos, \emph{The quasimetrization problem in the (bi)topological spaces}, Int. J. Math. Math. Sci. 2007, Art. ID 76904, 19pp. 

\bibitem[AHS]{AHS}  J.Ad\'amek, H. Herrlich, G. E. Strecher \emph{Abstract and Concrete Categories,
The Joy of Cats}, Wiley, 1990.

\bibitem[Be]{Be} G. Beer, \emph{On metric boundedness structures}, Set-Valued Analysis 7 (1999), 195--208.

\bibitem[CL]{CL} E. Colebunders, R. Lowen, \emph{Metrically generated theories}, Proc. Amer. Math. Soc. 133(5) (2004), 1547--1556.

\bibitem[Cruz]{Cruz} O. De la Cruz, \emph{Finiteness and choice}, Fund. Math. 173 (2002), 57--76.

\bibitem[DK]{DK} H. Delfs, M. Knebusch, \emph{Locally Semialgebraic Spaces}, Lectures Notes in Math. 1173, Springer 1985.

\bibitem[En]{En} R. Engelking, \emph{General Topology}, Sigma Series in Pure Mathematics 6, Heldermann, Berlin 1989. 

\bibitem[FL]{FL} P. Fletcher, W. F. Lindgren, \emph{Quasi-Uniform Spaces}, Marcel Dekker, New York 1982.

\bibitem[Gut]{Gut} G. Gutierres, \emph{The axiom of countable choice in topology}, Thesis, Dep. of Math, Coimbra Univ. 2004.

\bibitem[GM]{GM} I. Garrido, A. S. Mero\~{n}o, \emph{Uniformly metrizable bornologies}, J. of Convex Analysis 20(1)   (2013), 285--299.

\bibitem[Her]{Her} H. Herrlich, \emph{Axiom of Choice},
Springer-Verlag, Berlin-Heidelberg 2006. 

\bibitem[H-N]{H-N} H. Hogbe-Nlend, \emph{Bornologies and Functional Analysis}, North-Holland, Amsterdam 1977.

\bibitem[HR]{HR} P. Howard, J. E. Rubin, \emph{Consequences of the Axiom of Choice}, Math. Surveys Monogr. 59, Amer. Math. Soc. 1998. 

\bibitem[Hu]{Hu} S. T. Hu, \emph{Boundedness in a topological space}, J.  Math. Pures. Appl. 28 (1949), 287--320.

\bibitem[J1]{Jech1} T. Jech, \emph{The Axiom of Choice},
North-Holland, Amsterdam 1973. 

\bibitem[J2]{Jech2} T. Jech, \emph{Set Theory},
Springer-Verlag, Berlin-Heidelberg 2002. 

\bibitem[Kel]{Kel} J. C. Kelly, \emph{Bitopological spaces}, Proc. London Math. Soc. 13 (1963), 71--89.

\bibitem[Kn]{Kn} M.  Knebusch, \emph{Weakly Semialgebraic Spaces}, Lect. Notes Math. 1367, Springer-Verlag, New York 1989.

\bibitem[Ku1]{Ku1}  K. Kunen, \emph{Set Theory},
North-Holland, Amsterdam 1980. 

\bibitem[Ku2]{Ku2}  K. Kunen, \emph{The Foundations of
Mathematics}, Individual Authors and College Publications, London 2009. 

\bibitem[P1]{Pie1} A. Pi\k{e}kosz, \emph{On generalized
topological spaces I}, Annales Polonici Mathematici 107(3)   (2013),
217--241. 

\bibitem[P2]{Pie2} A. Pi\k{e}kosz, \emph{On generalized
topological spaces II}, Annales Polonici Mathematici 108(2) (2013),
185--214. 

\bibitem[PW]{PW} A. Pi\k{e}kosz, E. Wajch \emph{Compactness and compactifications in generalized topology}, preprint, arXiv:1402.1286 [math.GN], submitted.

\bibitem[Sal]{Sal} S. Salbany, \emph{Bitopological spaces, compactifications and completions}, Math. Monographs, Univ. of Cape Town, 1974.

\bibitem[Vr1]{Vr1} T. Vroegrijk, \emph{Pointwise bornological spaces}, Topology Appl. 156 (2009), 2019--2027.

\bibitem[Vr2]{Vr2} T. Vroegrijk, \emph{On realcompactifications defined by bornologies}, Acta Math. Hungar. 133(4)   (2011), 387--395.

\bibitem[W]{W} E. Wajch, \emph{Conditions of separation for quasi-pseudometrics}, Folia Math. 16(1) (2009), 45--55.

\bibitem[WJ]{WJ} R. Williamson, L. Janos, \emph{Constructing metrics with the Heine-Borel property}, Proc. Amer. Math. Soc. 100 (1987), 567--573. 

 
\end{thebibliography}
\end{document}